\RequirePackage{fix-cm}
\documentclass[smallextended]{svjour3}       
\smartqed  
\usepackage{graphicx}

\usepackage{amsmath}
\usepackage{amssymb}
\usepackage{amsfonts}
\usepackage{graphicx}
\usepackage{times}
\usepackage{mathptmx} 
\usepackage{datetime}

\usepackage{tikz}
\usetikzlibrary{arrows}

\usepackage{datetime}
\usepackage{times} 
\usepackage{framed} 
\usepackage{graphicx} 
\usepackage{amsmath} 
\usepackage{amssymb}  

\newcommand{\sinhc}{\mathrm{sinhc}}
\newcommand{\sinc}{\mathrm{sinc}}

\newcommand{\ad}{\mathrm{ad}}           
\def\<{\leqslant}           
\def\>{\geqslant}           

\def\d{\partial}
\def\wh{\widehat}
\def\wt{\widetilde}

\def\Re{{\rm Re}}   
\def\Im{{\rm Im}}   
\def\rprod{\mathop{\overrightarrow{\prod}}}

\def\conv{\mathop{\overrightarrow{\circledast}}}

\def\cH{\mathcal{H}}   
\def\mR{{\mathbb R}}    
\def\mC{\mathbb{C}}    

\def\Tr{\mathrm{Tr}}       
\def\rT{{\rm T}}        
\def\bS{{\mathbf S}}

\def\bE{\mathbf{E}}    

\def\bK{\mathbf{K}}    

\def\[[[{[\![\![}
\def\]]]{]\!]\!]}

\def\bra{{\langle}}
\def\ket{{\rangle}}

\def\Bra{\left\langle}
\def\Ket{\right\rangle}

\def\re{{\rm e}}        
\def\rd{{\rm d}}        

\def\bA{{\mathbf A}}

\def\bJ{\mathbf{J}}

\def\x{\times}
\def\ox{\otimes}

\def\fB{\mathfrak{B}}

\def\fF{\mathfrak{F}}
\def\fI{\mathfrak{I}}

\def\fH{\mathfrak{H}}
\def\fS{\mathfrak{S}}

\def\cW{\mathcal{W}}

\def\cD{\mathcal{D}}

\def\cM{\mathcal{M}}

\def\sP{\mathsf{P}}

\def\cG{\mathcal{G}}
\def\cI{\mathcal{I}}
\def\cP{\mathcal{P}}

\def\cov{\mathbf{cov}}
\def\var{\mathbf{var}}

\def\eps{\epsilon}

\def\Ups{\Upsilon}

\journalname{Applied Mathematics and Optimization}

\begin{document}

\title{Multi-point Gaussian States,
Quadratic-exponential Cost Functionals, and Large Deviations Estimates
for Linear Quantum Stochastic Systems$^*$
\thanks{\noindent
$^*$This work is supported
by the Air Force Office of Scientific Research (AFOSR) under agreement number FA2386-16-1-4065.
The authors are with the Research School of Engineering, College of Engineering and Computer Science, Australian National University, Canberra, Acton, ACT 2601, Australia.
}
}

\titlerunning{Quadratic-exponential Functionals and Large Deviations Estimates
for Linear Quantum Systems}        

\author{Igor~G.~Vladimirov \and
Ian~R.~Petersen \and
Matthew~R.~James}

\authorrunning{Igor~G.~Vladimirov,
Ian~R.~Petersen,
Matthew~R.~James} 

\institute{Igor G. Vladimirov \at
              \email{igor.g.vladimirov@gmail.com}           
           \and
           Ian R. Petersen \at
           \email{i.r.petersen@gmail.com}
           \and
           Matthew R. James \at
           \email{matthew.james@anu.edu.au}
}

\date{}

\maketitle

\begin{abstract}
This paper is concerned with risk-sensitive performance analysis for linear quantum stochastic systems interacting with external bosonic fields.
We consider a cost functional in the form of the exponential moment of the integral of a quadratic polynomial of the system variables over a bounded time interval. An integro-differential equation is obtained for the time evolution of this quadratic-exponential functional, which is compared with the original quantum risk-sensitive performance criterion employed previously  for measurement-based quantum control and filtering problems. Using multi-point Gaussian quantum states  for the past history of the system variables and their first four moments, we discuss a quartic approximation of the cost
functional and its infinite-horizon asymptotic behaviour. The computation of the asymptotic growth rate of this approximation is reduced to solving two algebraic Lyapunov equations.
We also outline further approximations of the cost functional, based on higher-order  cumulants and their growth rates,  together with large deviations estimates.
For comparison, an auxiliary classical Gaussian Markov  diffusion process is considered in a complex Euclidean space which reproduces the quantum system variables at the level of covariances but has different higher-order moments relevant to the risk-sensitive criteria.
The results of the paper are also demonstrated by a numerical example and may find applications
 to coherent quantum risk-sensitive control problems, where the plant and controller form a fully quantum closed-loop system, and other settings with nonquadratic cost functionals.
\end{abstract}

%

\section{Introduction}
\label{intro}

The main theme of the present paper is a class of risk-sensitive performance criteria for linear quantum stochastic systems. Such systems, also referred to as open quantum harmonic oscillators (OQHOs) \cite{GZ_2004}, play the role of building blocks in linear quantum systems theory \cite{P_2017}. This paradigm is part of the broader area of quantum filtering and control (see, for example, \cite{B_1983,B_2010,BH_2006,BVJ_2007,EB_2005,J_2004,J_2005,JNP_2008,NJP_2009,WM_2010})
 which is concerned with achieving certain dynamic properties for open quantum systems interacting with surroundings such as classical measuring devices, other quantum systems or external quantum fields. In particular, these properties may include stability, optimality and robustness. In contrast to their classical counterparts, quantum systems are equipped with noncommuting operator-valued variables whose evolution obeys the laws of quantum mechanics \cite{LL_1991,M_1998,S_1994} and statistical characteristics are described in terms of quantum probability \cite{M_1995}. The applications include, for example, artificially engineered systems for influencing the state of matter at atomic scales through its interaction with nonclassical light in quantum optics \cite{WM_2008} and quantum computing \cite{NC_2000}.

A unified language for the modelling of open quantum systems, interacting with the environment, is provided by Hudson-Parthasarathy  quantum stochastic differential equations (QSDEs) \cite{HP_1984,P_1992,P_2015} (see, also, the review paper \cite{H_1991} and references therein) which govern the system variables in the Heisenberg picture of quantum dynamics. The QSDEs are driven by quantum Wiener processes on a symmetric Fock space \cite{P_1992,PS_1972} which represent external bosonic fields (such as quantised electromagnetic radiation). The structure of these quantum stochastic dynamics depends on the energetics of the system-field interaction and the self-energy of the system. These are captured by the system-field coupling operators and the system Hamiltonian, with the latter specifying the internal dynamics of the system in isolation from the environment. This approach to open quantum systems and their interconnections is widely used in quantum control  \cite{B_1983,EB_2005,J_2005,JNP_2008}, where it is also combined with the theory of quantum feedback networks \cite{GJ_2009,JG_2010} (see also \cite{ZJ_2012} and references therein). 
In the case of OQHOs, the system variables satisfy canonical commutation relations (similar to those for the quantum-mechanical position and momentum operators  \cite{M_1998}), the Hamiltonian is a quadratic function and the coupling operators are linear functions of the variables. The linearity of the resulting QSDEs makes them tractable in many respects including the dynamics of the first two moments of the system variables and preservation of the Gaussian nature of system states (provided the fields are in the vacuum state). Due to the linear-Gaussian quantum dynamics, such systems resemble Gaussian Markov diffusion processes generated by classical linear SDEs \cite{KS_1991}.  This analogy is exploited in the quantum counterparts \cite{EB_2005,MP_2009,NJP_2009} of the classical linear quadratic Gaussian (LQG) control and filtering problems \cite{AM_1979,KS_1972} (and also $\cH_{\infty}$-control settings \cite{JNP_2008}). The optimality in these approaches is understood as the minimization of mean square cost functionals which are concerned with second-order moments of the system variables at one instant in time (or the integrals of such moments in finite-horizon formulations \cite{VP_2011b}). A more general class of cost functionals for quantum systems, which use one-point averaging of  nonquadratic functions of system variables,   is considered in \cite{P_2014}.

Qualitatively different are the quantum risk-sensitive performance criteria \cite{J_2004,J_2005} (see also \cite{DDJW_2006,YB_2009}) which were used previously  for measurement-based quantum control and filtering problems. They employ weighted mean square values of time-ordered exponentials satisfying operator differential equations (where noncommutativity of quantum variables plays its part). Similarly to their classical predecessors \cite{BV_1985,J_1973,W_1981} and in contrast to the mean square values of the system variables themselves, the quantum risk-sensitive cost functionals involve multi-point quantum states and higher-order moments for the system variables at different instants. In addition to being a challenge from a purely theoretical point of view, the computation and minimization of such functionals is also of practical importantance.  In fact, this approach leads to controllers and filters which secure a more conservative behaviour of the system not only in terms of the one-point second-order moments of the system variables but also their higher-order multi-point moments.  Also, there are connections (though more limited in the quantum case) between the risk-sensitive criteria and robustness with respect to statistical uncertainty in the driving noise described in terms of relative entropy (see, for example,   \cite{J_2004,YB_2009} and references therein).

A quadratic-exponential functional (QEF), considered in the present paper, is organised as the exponential moment of the integral of a quadratic form of the system variables over a bounded time interval. It differs from the original quantum risk-sensitive criterion \cite{J_2004,J_2005}, mentioned above,  in that a weighted mean square value of the time-ordered exponential is replaced with the mean value of the operator exponential of the integral. At the same time, the cost functional studied here is, in fact, a more straightforward extension of its classical predecessors \cite{BV_1985,J_1973,W_1981} to the quantum case and also imposes penalty both on the second and higher-order moments of the system variables. As before, the shift towards the higher-order moments is controlled by a risk-sensitivity parameter. These features motivate  the study of such cost functionals for quantum systems.

To this end, we obtain an integro-differential equation for the time evolution of the QEF and compare it with the original quantum risk-sensitive performance criterion.  Assuming that the OQHO is driven by vacuum fields and using the multi-point Gaussian quantum states for the system variables at different instants and their first four moments, we study a quartic approximation of the cost functional and its infinite-horizon asymptotic behaviour. The  computation of the asymptotic growth rate for this approximation is reduced to solving two algebraic Lyapunov equations. This is a quantum extension of similar results of \cite{VP_2010b} on classical linear stochastic systems. However, we also discuss an auxiliary classical Gaussian Markov diffusion  process in a complex Euclidean space which reproduces the quantum system variables at the covariance function level but has different higher-order moments relevant to the risk-sensitive criteria. In addition to the quartic approximation of the QEF, we also outline its higher-order approximations and their asymptotic growth  rates along with a large deviations estimate   of the Cramer type \cite{DE_1997,S_1996}.

The paper is organised as follows.
Section~\ref{sec:class} specifies the class of linear quantum stochastic systems being considered
 and provides background material.
Section~\ref{sec:multi} discusses multi-point Gaussian quantum states associated with the system variables at different instants.
Section~\ref{sec:QEF} defines the quadratic-exponential functional
and establishes an integro-differential equation for its time evolution.
Section~\ref{sec:quart}  discusses a quartic approximation of the QEF and its state-space computation.
Section~\ref{sec:cumul} outlines further approximations of the QEF using higher-order cumulants.
Section~\ref{sec:large} establishes an upper bound for the cumulants and provides a related large deviations estimate.
Section~\ref{sec:CCCP} considers a correspondence between vectors of self-adjoint quantum variables and classical random vectors at the level of covariances and discusses its violation in regard to higher-order moments.
Section~\ref{sec:numer} provides a numerical example which demonstrates the quartic approximation of the QEF for a two-mode oscillator.
Section~\ref{sec:conc} makes concluding remarks.
Appendices \ref{sec:quadcom}, \ref{sec:covquad} and \ref{sec:conv}  provide auxiliary lemmas on computing the commutator and covariance of quadratic functions of Gaussian quantum variables, and also on averaging in a class of convolution-like integrals of matrix-valued functions.

\section{Linear quantum stochastic systems}\label{sec:class}

We consider a quantum system whose dynamic variables $X_1, \ldots, X_n$ are time-varying self-adjoint operators, defined on a dense domain in a complex separable Hilbert space $\fH$ and satisfying the Weyl  canonical commutation relations (CCRs)
\begin{equation}
\label{CCR}
    \cW_u \cW_v
    =
    \re^{i v^{\rT}\Theta u}
    \cW_{u+v}
\end{equation}
for all $u,v\in \mR^n$. Here, $i:= \sqrt{-1}$ is the imaginary unit, and
$\Theta:= (\theta_{jk})_{1\< j,k\< n}$ is a nonsingular real antisymmetric matrix of order $n$ (which is assumed to be even). Also,
\begin{equation}
\label{cW}
  \cW_u
  :=
  \re^{iu^{\rT} X}
\end{equation}
is the unitary Weyl operator \cite{F_1989} parameterised by a vector $u:=(u_k)_{1\< k\< n} \in \mR^n$ whose entries specify the linear combination $u^{\rT} X = \sum_{k=1}^n u_k X_k$ of the system variables assembled into the vector
\begin{equation}
\label{X}
    X:=
    \begin{bmatrix}
        X_1\\
        \vdots\\
        X_n
    \end{bmatrix}
\end{equation}
(vectors are organised  as columns unless indicated otherwise).
Due to self-adjointness of $u^{\rT}X$, the adjoint of the Weyl operator in (\ref{cW}) satisfies $\cW_u^{\dagger} = \cW_{-u}$.
 Also, (\ref{CCR}) implies that
\begin{equation}
\label{cWcom}
    [\cW_u, \cW_v]
    =
    -2i\sin(u^{\rT}\Theta v)
    \cW_{u+v},
\end{equation}
where $[\xi,\eta]:=\xi\eta - \eta \xi$ is the commutator of linear operators.
In view of (\ref{cWcom}),
the Heisenberg  infinitesimal form of the Weyl CCRs (\ref{CCR}) is described by the commutator matrix
\begin{align}
\nonumber
    [X, X^{\rT}]
      & :=    ([X_j,X_k])_{1\< j,k\< n}\\
\label{Theta}
     & =    XX^{\rT}- (XX^{\rT})^{\rT}
     =
     2i \Theta \ox \cI_{\fH}
\end{align}
(on a dense domain in $\fH$),
where $\ox$ is the tensor product, and $\cI_{\fH}$ is the identity operator on $\fH$ (the subscript will sometimes be omitted for the sake of brevity). The transpose $(\cdot)^{\rT}$ acts on matrices with operator-valued entries as if the latter were scalars. By a standard convention on linear operators, $\Theta \ox \cI_{\fH}$ in (\ref{Theta}) will be identified with the matrix $\Theta$.

For example, in the case when the system variables consist of the pairs of conjugate quantum mechanical position $q_k$ and momentum $p_k:= -i\d_{q_k}$ operators (with the Planck constant being appropriately normalised) \cite{M_1998,S_1994}, with $1\< k\< \frac{n}{2}$, they are defined on the Schwartz space \cite{V_2002} of rapidly decaying complex-valued functions on $\mR^{n/2}$ which is dense in the Hilbert space $L^2(\mR^{n/2})$ of square integrable functions (and is invariant under those operators). In this case, $[q_j,p_k] = i\delta_{jk}$ for all $j,k=1,\ldots, \frac{n}{2}$, where $\delta_{jk}$ is the Kronecker delta,  and the CCR matrix $\Theta$ of the system variables $q_1, \ldots, q_{n/2}, p_1, \ldots, p_{n/2}$ takes the form
\begin{equation}
\label{ThetaJ}
    \Theta
    =
    \frac{1}{2}
    \begin{bmatrix}
        0 & I_{n/2}\\
        -I_{n/2} & 0
    \end{bmatrix}
    =
    \frac{1}{2}
    \bJ \ox I_{n/2},
\end{equation}
where $\ox$ is the Kronecker product of matrices, $I_r$ denotes the identity matrix of order $r$, and use is made of the matrix
\begin{equation}
\label{bJ}
    \bJ
    :=
    \begin{bmatrix}
        0 & 1\\
        -1 & 0
    \end{bmatrix}
\end{equation}
which spans the space of antisymmetric matrices of order $2$. In fact, any CCR matrix can be brought to the form (\ref{ThetaJ}) by an appropriate linear transformation of the system variables.

The system under consideration is organised as a linear quantum stochastic system which models an open quantum harmonic oscillator (OQHO) \cite{EB_2005,GZ_2004} with $\frac{n}{2}$ modes which interacts with an $m$-channel external bosonic field, where $m$ is even. The energetics of the system itself and its interaction with the fields is specified by the system Hamiltonian $H$ and the system-field coupling operators $L_1, \ldots, L_m$ which are, respectively, quadratic and linear functions of the system variables:
\begin{equation}
\label{H_L}
    H
    :=
    \frac{1}{2}
    \sum_{j,k=1}^{n}
    r_{jk}X_jX_k
    =
    \frac{1}{2}
    X^{\rT} R X,
    \qquad
    L
    :=
    \begin{bmatrix}
        L_1\\
        \vdots\\
        L_m
    \end{bmatrix}
    =
    MX.
\end{equation}
Here, $R:= (r_{jk})_{1\< j,k\< n} \in \mR^{n\x n}$ and $M \in \mR^{m\x n}$ are given matrices, with $R$ being symmetric, which will be referred to as the energy and coupling matrices.   The evolution of the system variables in time $t\>0$ is governed by a linear Hudson-Parthasarathy QSDE \cite{HP_1984,P_1992}
\begin{equation}
\label{dX}
    \rd X
    =
    \cG(X)\rd t - i[X,L^{\rT}] \rd W
    =
    A X \rd t+ B \rd W
\end{equation}
whose structure (including its linearity in the case being considered)
are clarified below.  This QSDE is driven by the vector
$$
  W:=
  \begin{bmatrix}
        W_1\\
        \vdots\\
        W_m
  \end{bmatrix}
$$
of quantum Wiener processes $W_1, \ldots, W_m$ which are time-varying self-adjoint operators on a symmetric Fock space $\fF$ \cite{P_1992,PS_1972}. These operators  represent the external fields and have a complex positive semi-definite Hermitian Ito matrix $\Omega:= (\omega_{jk})_{1\< j,k\< m}$:
\begin{equation}
\label{WW}
    \rd W \rd W^{\rT}
    =
    \Omega \rd t,
    \qquad
    \Omega := I_m + iJ.
\end{equation}
Here,
\begin{equation}
\label{J}
        J
        :=
       \begin{bmatrix}
           0 & I_{m/2}\\
           -I_{m/2} & 0
       \end{bmatrix}
       =
        \bJ \ox I_{m/2}
\end{equation}
is an orthogonal real antisymmetric matrix of order $m$ (so that $J^2=-I_m$), which specifies CCRs for the quantum Wiener processes as $[\rd W, \rd W^{\rT}] = 2iJ\rd t$, and use is made of the matrix $\bJ$ from (\ref{bJ}).

The Hilbert space $\fH$, which provides a common domain for the action of the system and external field variables,  has the tensor-product structure
\begin{equation}
\label{fH}
    \fH
    :=
    \fH_0 \ox \fF,
\end{equation}
where $\fH_0$ is the initial system space on which the system variables  $X_1(0), \ldots, X_n(0)$ are defined. The matrices $A\in \mR^{n\x n}$ and $B\in \mR^{n\x m}$ in (\ref{dX}) are expressed as
\begin{equation}
\label{A_B}
    A
    :=
    2\Theta (R + M^{\rT} J M)
    =
    2\Theta R - \frac{1}{2} BJB^{\rT} \Theta^{-1},
    \qquad
    B
    := 2\Theta M^{\rT}
\end{equation}
in terms of the energy and coupling matrices $R$ and $M$ from (\ref{H_L})
and satisfy the physical realizability (PR) condition \cite{JNP_2008,SP_2012}:
\begin{align}
\label{APR}
    A \Theta + \Theta A^{\rT} + BJB^{\rT} & = 0.
\end{align}
Also, $\cG$ in (\ref{dX}) denotes the  Gorini-Kossakowski-Sudarshan-Lindblad generator \cite{GKS_1976,L_1976} (see also \cite{A_2000}), whose role is similar to that of the infinitesimal generators of classical Markov diffusion processes \cite{KS_1991}. More precisely, $\cG$ is a linear superoperator which specifies the drift term in the evolution of a system operator $\xi$ (a function of the system variables $X_1,\ldots,X_n$) as
\begin{equation}
\label{xiQSDE}
    \rd \xi
    =
    \cG(\xi)\rd t - [\xi, h^{\rT}]\rd W.
\end{equation}
Furthermore, $\cG$ takes into account both the internal dynamics of the system and its interaction with the external fields by acting on $\xi$ as
\begin{equation}
\label{cG}
    \cG(\xi)
    :=
    i[H,\xi]
    +
    \cD(\xi),
\end{equation}
where $H$ is the system Hamiltonian in (\ref{H_L}).  The last term in (\ref{cG}) is associated with the system-field interaction and involves the decoherence superoperator $\cD$ which acts as
\begin{equation}
\label{cD}
    \cD(\xi)
    :=
    \frac{1}{2}
    \sum_{j,k=1}^m
    \omega_{jk}
    \left(
        L_j[\xi,L_k] + [L_j,\xi]L_k
    \right).
\end{equation}
Here, $\omega_{jk}$ are the entries of the quantum Ito matrix $\Omega$ in (\ref{WW}), and $L_1, \ldots, L_m$ are the coupling operators from (\ref{H_L}).
In particular, by substituting $L$ from (\ref{H_L}) into (\ref{cD}) and using the CCRs (\ref{Theta}), the entrywise application  of $\cD$ to the vector $X$ of the system variables in (\ref{X}) leads to
\begin{align}
\nonumber
    \cD(X)
    & :=
    \frac{1}{2}
    \sum_{j,k=1}^m
    \omega_{jk}
    \left(
        L_j[X,L_k] + [L_j,X]L_k
    \right)\\
\label{cD1}
    & = 2\Theta M^{\rT} J M X
     =
      - \frac{1}{2} BJB^{\rT} \Theta^{-1} X,
\end{align}
which gives rise to the field-related term $2\Theta M^{\rT}JM = - \frac{1}{2} BJB^{\rT} \Theta^{-1}$ in the matrix $A$ in (\ref{A_B}). The other part $i[H,X] = 2\Theta R X$ of the drift vector $AX$ in the QSDE (\ref{dX}) comes from the quadratic nature of the Hamiltonian $H$ in (\ref{H_L}) in combination with the CCRs (\ref{Theta}) and describes the internal dynamics of the system variables which they would have if the system were isolated from the environment. Also, the representation of the  dispersion matrix $B = -i[X,L^{\rT}]$ of the QSDE (\ref{dX}) in (\ref{A_B})  follows from the linear dependence of the coupling operators in (\ref{H_L}) on the system variables and the CCRs (\ref{Theta}).

In contrast to classical SDEs, the specific structure of the drift and diffusion terms of the QSDE (\ref{xiQSDE}) (and its particular case (\ref{dX})),  described by (\ref{cG})--(\ref{cD1}),  comes from  the evolution
\begin{equation}
\label{xiuni}
    \xi(t)
    =
    U(t)^{\dagger} (\xi(0)\ox \cI_{\fF}) U(t)
\end{equation}
of the system operator $\xi$, with $\xi(0)$ acting on the initial system space $\fH_0$.  Here, $U(t)$ is a time-varying unitary operator which acts on the system-field space $\fH$ in (\ref{fH}) and satisfies another QSDE
\begin{equation}
\label{dU}
    \rd U(t)
    =
    -U(t) \Big(i(H(t)\rd t + L(t)^{\rT} \rd W(t)) + \frac{1}{2}L(t)^{\rT}\Omega L(t)\rd t\Big),
\end{equation}
with $U(0)=\cI_{\fH}$. In accordance with the stochastic flow (\ref{xiuni}), the current Hamiltonian and coupling operators  in (\ref{dU}) are given by
$
    H(t)
     =
    U(t)^{\dagger} (H(0)\ox \cI_{\fF}) U(t)
$ and
$
    L(t)
    =
    U(t)^{\dagger} (L(0)\ox \cI_{\fF}) U(t)
$
and retain their quadratic-linear  dependence (\ref{H_L}) on the system variables $X_1(t), \ldots, X_n(t)$. This property follows from the fact that the map $ \zeta\mapsto U(t)^{\dagger} \zeta U(t)$, which acts on operators  $\zeta$ on $\fH$  (and applies to vectors of such operators entrywise),  is a unitary similarity transformation.

In view of (\ref{dU}), the unitary operator $U(t)$ reflects an accumulated effect from internal driving forces of the system and its interaction with the external fields over the time interval $[0,t]$ and is adapted in the sense that  it acts effectively on the subspace $\fH_0\ox \fF_t$, where $\{\fF_t:\, t\>0\}$ is the Fock space filtration.
The QSDE (\ref{xiQSDE}) follows from (\ref{xiuni}) and (\ref{dU}) due to the quantum Ito formula \cite{HP_1984,P_1992} combined with (\ref{WW}), unitarity of $U$ and commutativity between the forward Ito increments $\rd W(t)$ and adapted processes (including $U$) considered at time $s\< t$. This commutativity is a consequence of the continuous tensor-product structure  \cite{PS_1972} of the Fock space $\fF$.

The QSDEs (\ref{dX}), (\ref{xiQSDE}) and (\ref{dU}) correspond to a particular yet important scenario of quantum stochastic dynamics, where there is no photon exchange between the fields, and the  scattering matrix is an identity matrix, which effectively eliminates the gauge processes from consideration. The presence of  these processes and more general scattering matrices \cite{HP_1984,P_1992} affects the dynamics of the unitary operator $U$ and is taken into account in QSDEs which are used in the quantum feedback network theory \cite{GJ_2009,JG_2010}.

Due to the linearity of the QSDE (\ref{dX}), the OQHO is employed extensively  as a basic model in linear quantum filtering and control \cite{JNP_2008,NJP_2009,P_2017} with quadratic performance criteria. If the initial system variables have finite second moments, that is,
\begin{equation}
\label{finsec}
    \bE
    \big(
        X(0)^{\rT}X(0)
    \big)
    =
    \sum_{k=1}^n
    \bE
    \big(
        X_k(0)^2
    \big)
    <
    +\infty,
\end{equation}
then, similarly to the classical case, this property is preserved in time by the linear dynamics (\ref{dX}) and leads to the finite limit values of the first and second moments:
\begin{equation}
\label{EXX}
    \lim_{t\to +\infty}
    \bE X(t)
    =
    0,
    \qquad
    \lim_{t\to +\infty}
    \bE
    \big(
        X(t)X(t)^{\rT}
    \big)
    =
    P+i\Theta,
\end{equation}
provided the matrix $A$ in (\ref{A_B}) is Hurwitz.
Here, $\bE\xi:= \Tr(\rho \xi)$ is the expectation of a quantum variable $\xi$ over the density operator \begin{equation}
\label{rho}
    \rho:= \varpi\ox \upsilon,
\end{equation}
where $\varpi$ is the initial system state on $\fH_0$, and $\upsilon$ is the vacuum state of the input bosonic fields on the Fock space $\fF$. The matrix $P$ in (\ref{EXX}) is the infinite-horizon controllability Gramian
\begin{equation}
\label{P}
  P
  :=
  \int_0^{+\infty}
  \re^{tA}
  BB^{\rT}
  \re^{tA^{\rT}}
  \rd t
\end{equation}
of the matrix pair $(A,B)$,  which is a unique solution  of the following algebraic Lyapunov equation (ALE)
due to the matrix $A$ being Hurwitz:
\begin{equation}
\label{PALE}
    AP + PA^{\rT} + BB^{\rT} = 0.
\end{equation}
The first two moments of the system variables, mentioned above, are part of the averaged behaviour of the system at a particular instant and do not provide information on multi-point quantum correlations at different moments of time.

\section{Multi-point Gaussian quantum states}\label{sec:multi}

We will first revisit the statistical properties of the system variables $X_1(t), \ldots, X_n(t)$ at one point in time $t\>0$. These properties are encoded by the one-point quasi-characteristic function (QCF)  $\Phi: \mR_+\x \mR^n\to \mC$ defined by (see, for example, \cite{CH_1971})
\begin{equation}
\label{Phi}
    \Phi(t,u):= \bE \cW_u(t)
\end{equation}
in terms of the Weyl operator (\ref{cW}) associated with $X(t)$, with the averaging over the system-field state (\ref{rho}). Also, for any $t\>0$,
we denote by
\begin{equation}
\label{Sigma}
    \Sigma(t)
    :=
    \int_0^t
    \re^{(t-s)A}
    BB^{\rT}
    \re^{(t-s)A^{\rT}}
    \rd s
    =
    \int_0^t
    \re^{sA}
    BB^{\rT}
    \re^{sA^{\rT}}
    \rd s
\end{equation}
the finite-horizon (over the time interval $[0,t]$) controllability Gramian of the pair $(A,B)$,
satisfying the Lyapunov ODE
$$
    \dot{\Sigma} = A\Sigma + \Sigma A^{\rT} + BB^{\rT},
$$
with
$
    \Sigma(0) = 0
$. Since the matrix $A$ is assumed to be Hurwitz,  the function $\Sigma$ is related to the matrix $P$ in (\ref{P}) by
\begin{equation}
\label{PSigma}
    \Sigma(t)
    =
    P - \re^{tA}P\re^{tA^{\rT}}
    \to P,
    \qquad
    {\rm as}\
    t\to +\infty.
\end{equation}
The following lemma, which is given here for completeness,  shows that the evolution of the QCF under the linear QSDE  is identical to that for  classical Gaussian Markov diffusion processes (such as the Ornstein-Uhlenbeck process \cite{KS_1991}) produced by linear SDEs.

\begin{lemma}
Suppose the OQHO, governed by the linear QSDE (\ref{dX}), is driven by the input fields in the vacuum state. Then the QCF $\Phi$ for the system variables  in (\ref{Phi}) satisfies a linear functional equation
\begin{equation}
\label{Phist}
    \Phi(t,u)
     =
     \Phi\big(s,\re^{(t-s)A^{\rT}}u\big)
     \re^{-\frac{1}{2}\|u\|_{\Sigma(t-s)}^2}
\end{equation}
for any two moments of time $t\> s\> 0$ and all $u \in \mR^n$. Here, the function $\Sigma$ is given by (\ref{Sigma}), and $\|v\|_K:= \sqrt{v^{\rT}K v} = |\sqrt{K}v|$ denotes a weighted Euclidean semi-norm of a vector $v$ associated with a real  positive semi-definite symmetric matrix $K$.
\hfill$\square$
\end{lemma}
\begin{proof}
The solution of the linear QSDE (\ref{dX}) admits the decomposition
\begin{equation}
\label{XY}
    X(t)   = \re^{(t-s)A}X(s) + Y(s,t)
\end{equation}
for any $t\>s\>0$, where $Y(s,t)$ is an auxiliary vector of self-adjoint operators given by
\begin{equation}
\label{Y}
    Y(s,t)
     :=
    \int_s^t
    \re^{(t-\tau)A}
    B \rd W (\tau).
\end{equation}
Since the increments of the quantum Wiener process $W$ commute with the current values of adapted processes, then
\begin{equation}
\label{XYcom}
    [X(s),Y(s,t)^{\rT}] = \int_s^t [X(s),\rd W(\tau)^{\rT}] B^{\rT}\re^{(t-\tau)A^{\rT}} = 0
\end{equation}
for all $t\> s\> 0$. In fact, the entries of $X(s)$ and $Y(s,t)$ act on orthogonal subspaces of the system-field Hilbert space $\fH$ in (\ref{fH}). Furthermore, since the system and fields are in the product state (\ref{rho}), with the fields being in the vacuum state, then  $X(s)$ and $Y(s,t)$ are statistically independent. In combination with (\ref{XY}) and (\ref{XYcom}), this independence implies that
\begin{align}
\nonumber
    \Phi(t,u)
    & =
    \bE \re^{iu^{\rT}(\re^{(t-s)A}X(s) + Y(s,t))}\\
\nonumber
    & =
    \bE \re^{iu^{\rT}\re^{(t-s)A}X(s)}\Psi(s,t,u)\\
\label{Phist1}
    & =
    \Phi\big(s,\re^{(t-s)A^{\rT}}u\big)
    \Psi(s,t,u),
\end{align}
where
\begin{equation}
\label{Psi}
    \Psi(s,t,u)
    :=
    \bE\re^{iu^{\rT}Y(s,t)}
\end{equation}
is the QCF for the vector $Y(s,t)$. Now, in the vacuum state, the input fields have the quasi-characteristic functional
\begin{equation}
\label{QCFvac}
    \bE \re^{i\int_0^t f(s)^{\rT}\rd W(s)}
    =
    \re^{-\frac{1}{2}\int_0^t |f(s)|^2 \rd s}
\end{equation}
for any $t\>0$ and any locally square integrable function $f:\mR_+\to \mR^m$. By substituting $Y$ from (\ref{XY}) into (\ref{Psi}) and using (\ref{QCFvac}), it follows that
\begin{align}
\nonumber
    \Psi(s,t,u)
    & :=
    \bE\re^{iu^{\rT}    \int_s^t
    \re^{(t-\tau)A}
    B \rd W (\tau)}\\
\nonumber
    & =
    \re^{-\frac{1}{2}
    \int_s^t
    |
    B^{\rT}
    \re^{(t-\tau)A^{\rT}}
    u    |^2 \rd \tau}\\
\label{Psi1}
    &
    =
    \re^{-\frac{1}{2}u^{\rT}\int_0^{t-s}
    \re^{\tau A} BB^{\rT} \re^{\tau A^{\rT}}\rd \tau u}
    =
     \re^{-\frac{1}{2}\|u\|_{\Sigma(t-s)}^2}     ,
\end{align}
where use is made of (\ref{Sigma}). Substitution of (\ref{Psi1}) into (\ref{Phist1}) leads to (\ref{Phist}).
\end{proof}

From (\ref{Phist}), it follows that if the initial state $\varpi$ of the system is Gaussian \cite{KRP_2010,PS_2015}   (that is, $\ln\Phi(0,u)$ is a quadratic function of $u \in \mR^n$), then so is its reduced quantum state at subsequent moments of time $t>0$. Furthermore, since the matrix $A$ is  Hurwitz,  the relations (\ref{PSigma}), (\ref{Phist}) and the continuity $\lim_{u\to 0} \Phi(s,u) = \Phi(s,0) = 1$  imply the pointwise convergence of the QCF:
\begin{equation}
\label{limPhi}
    \lim_{t\to +\infty}\Phi(t,u)
    =
    \lim_{t\to +\infty}
    \Big(
     \Phi\big(0,\re^{tA^{\rT}}u\big)
     \re^{-\frac{1}{2}\|u\|_{\Sigma(t)}^2}
     \Big)
     =
    \re^{-\frac{1}{2}\|u\|_P^2},
    \qquad
    u \in \mR^n,
\end{equation}
which holds regardless of whether the initial state is Gaussian (or whether (\ref{finsec}) is satisfied). The relation (\ref{limPhi}) is equivalent to the
weak convergence \cite{B_1968,CH_1971}  of the reduced system state to a unique invariant Gaussian quantum state with zero mean and the quantum covariance matrix $P+i\Theta $ specified by the matrix $P$ in (\ref{P}).

An infinite-dimensional extension of (\ref{Phist}) is the quasi-characteristic functional of the quantum process $X$ over the time interval $[0,t]$:
\begin{align}
\nonumber
  \wt{\Phi}(t, u)
   & :=
   \bE\re^{i\int_0^t X(s)^{\rT}\rd u(s)}\\
\nonumber
   & =
  \bE
  \re^{i\int_0^t (\re^{sA} X(0) + Y(0,s))^{\rT}\rd u(s)}\\
\label{Phit}
  & =
    \Phi
    \Big(
        0,
        \int_0^t
        \re^{s A^{\rT}}\rd u(s)
    \Big)
  \re^{-\int_0^t (\int_0^s C(s,\tau)\rd u(\tau))^{\rT} \rd u(s) },
\end{align}
which is computed (similarly to that of a classical Gaussian random process) for any $t\>0$ and any function $u:\mR_+\to \mR^n$ of locally bounded variation. The integrals in (\ref{Phit}) are understood as appropriate (operator and vector-matrix) versions of the Riemann-Stieltjes integral  \cite{Y_1980}. Also,
use is made of the two-point covariance matrix
\begin{align}
\nonumber
    C(s,\tau)
    & :=
    \Re \bE (Y(0,s)Y(0,\tau)^{\rT})\\
\label{EYY}
    & =
    \re^{(s-\tau)A}\Sigma(\tau)
    =
    C(\tau,s)^{\rT},
    \qquad
    s\> \tau \>0,
\end{align}
for the quantum process
$
    Y(0,s)
    =
    \int_0^s \re^{(s-\tau)A}B\rd W(\tau)
$
from (\ref{XY}) and  (\ref{Y}),
where the function $\Sigma$ is given by (\ref{Sigma}). The covariance function $C$ in (\ref{EYY}) is the kernel function of the quadratic form
\begin{align*}
    \bE
    \Big(
    \Big(
        \int_0^t
        Y(0,s)^{\rT}
        \rd u(s)
    \Big)^2
    \Big)
     & =
        \int_{[0,t]^2}
        \rd u(s)^{\rT}
        C(s,\tau)
        \rd u(\tau)\\
    & =
    2
    \int_0^t
    \Big(\int_0^s C(s,\tau)\rd u(\tau)\Big)^{\rT} \rd u(s).
\end{align*}
In application to a piece-wise constant function
$
    u(t) = \sum_{1\<k \<N:\ t_k \< t} v_k
$
with increments $v_1, \ldots, v_N \in \mR^n$ at an increasing sequence of moments of time  $0\< t_1 \< \ldots \< t_N$ (with an arbitrary $N=1,2,3,\ldots$), the quantity $\wt{\Phi}(t_N,u)$ in (\ref{Phit}) becomes the $N$-point QCF of the system variables:
\begin{align}
\nonumber
  \Phi_N(t_1,\ldots, t_N; v_1, \ldots,v_N)
  & :=
  \bE\re^{i\sum_{k=1}^N v_k^{\rT}X(t_k)}\\
\label{PhiN}
  & =
    \Phi
    \Big(
        0,
        \sum_{k=1}^N
        \re^{t_k A^{\rT}}
        v_k
    \Big)\re^{-\frac{1}{2}\sum_{j,k=1}^N
  v_j^{\rT}C(t_j,t_k)v_k}.
\end{align}
Although (\ref{PhiN}) is a particular case of the quasi-characteristic  functional (\ref{Phit}), it is also  possible to obtain (\ref{Phit}) from (\ref{PhiN})  by approximating the integrals 
with appropriate Riemann-Stieltjes  sums.
The multi-point QCFs $\Phi_N$ with different values of $N$  are related to each other by
\begin{equation}
\label{Phidiff2}
    \Phi_N (t_1, \ldots, t_{N-1},t_{N-1}; v_1, \ldots,v_N)
     =
     \Phi_{N-1}(t_1, \ldots, t_{N-1}; v_1, \ldots,v_{N-1}, v_{N-1}+v_N).
\end{equation}
Furthermore, a multi-point extension of (\ref{Phist}) is the recurrence relation
\begin{align}
\nonumber
    \Phi_N(t_1,\ldots, t_N; v_1, \ldots, v_N)
    & = \bE \re^{i\sum_{k=1}^{N-1}v_k^{\rT}X(t_k) + iv_N^{\rT}(\re^{(t_N-t_{N-1})A}X(t_{N-1}) + Y(t_{N-1},t_N))}\\
\nonumber
     & =
     \bE \re^{i\sum_{k=1}^{N-1}v_k^{\rT}X(t_k) + iv_N^{\rT}\re^{(t_N-t_{N-1})A}X(t_{N-1})}
     \bE\re^{iv_N^{\rT}Y(t_{N-1},t_N)}\\
\label{PhiNst}
     & =
     \Phi_{N-1}
     \big(
        t_1,\ldots, t_{N-1};
        v_1,\ldots, v_{N-1} + \re^{(t_N-t_{N-1})A^{\rT}}v_N
     \big)
    \re^{-\frac{1}{2}\|v_N\|_{\Sigma(t_N-t_{N-1})}^2}
\end{align}
which holds for all $N=2,3,4,\ldots$, where use is made of (\ref{Sigma}), (\ref{XY}) and (\ref{Phidiff2}). It follows from (\ref{PhiNst}) by induction  that
$
    \ln \tfrac{\Phi_N(t_1,\ldots, t_N; v_1, \ldots, v_N)}{\Phi\big(0,\sum_{k=1}^N \re^{t_kA^{\rT}}v_k\big)}
$
is a quadratic function of the vectors $v_1, \ldots, v_N\in \mR^n$ whose coefficients depend on the  time arguments $t_1, \ldots, t_N$ only through their differences $t_k-t_{k-1}$. Therefore, since the matrix $A$ is assumed to be Hurwitz, the limit
\begin{equation}
\label{limttt}
    \lim_{t\to +\infty}\Phi_N(t_1+t,\ldots, t_N+t; v_1, \ldots, v_N)
    =
    \re^{-\frac{1}{2}
    \sum_{j,k=1}^N
    v_j^{\rT} V(t_j-t_k)v_k}
\end{equation}
is the QCF of a multi-point Gaussian quantum state with zero mean and quantum covariance matrix whose real part is identical to the covariance matrix of a homogeneous Gaussian Markov diffusion process in $\mR^n$ considered at the moments of time $0, t_2-t_1, \ldots, t_N-t_1$. The covariance function $V$ of this auxiliary classical process in (\ref{limttt}) is computed as
\begin{equation}
\label{V}
    V(\tau)
    =
    V(-\tau)^{\rT}
    =
    \re^{\tau A} P,
    \qquad
    \tau \>0,
\end{equation}
where $P$ is the matrix from (\ref{P}). Due to the quantum nature of the setting under consideration,   the $\mC^{n\x n}$-valued function
\begin{equation}
\label{SVLambda}
    S(\tau)
    :=
    V(\tau) + i\Lambda(\tau)
    =
    S(-\tau)^*
\end{equation}
is a positive semi-definite Hermitian kernel (in the sense that so is the matrix $(S(\tau_j-\tau_k))_{1\< j,k\< N}$ for all $\tau_1, \ldots, \tau_N\> 0$ and $N=1,2,3,\ldots$). Here, $(\cdot)^*:= (\overline{(\cdot)})^{\rT}$ denotes the complex conjugate transpose, and
\begin{equation}
\label{Lambda}
    \Lambda(\tau)
    =
    -\Lambda(-\tau)^{\rT}
    =
    \re^{\tau A}\Theta,
    \qquad
    \tau \> 0,
\end{equation}
describes the two-point commutator matrix for the system variables:
\begin{equation}
\label{XsXtcomm}
    [X(s), X(t)^{\rT}]
    =
    2i\Lambda(s-t),
    \qquad
    s,t\> 0,
\end{equation}
where use is made of (\ref{XY})--(\ref{XYcom}). At the same time, $S$ in (\ref{SVLambda}) describes the covariance function of a classical stationary Gaussian Markov diffusion process in the complex Euclidean space $\mC^n$, which will be discussed in Section~\ref{sec:CCCP}.

As in the one-point case (\ref{limPhi}), the convergence of the multi-point QCFs in (\ref{limttt}) holds regardless of whether the initial system state is Gaussian (or whether (\ref{finsec}) is satisfied).
However, a Gaussian (but not necessarily invariant) initial  state gives rise to a multi-point Gaussian state associated with $X(t_1), \ldots, X(t_N)$ at any future moments of time $t_N\> \ldots \> t_1\> 0$ for any $N=1,2,3,\ldots$. The quantum covariance matrices for such states are specified by the two-point quantum covariance  function of the system variables given by
\begin{align}
\nonumber
    \cov(X(t),X(s))
    & :=
    \bE(X(t)X(s)^{\rT}) - \bE X(t)\bE X(s)^{\rT}\\
\label{Xcovts}
    & =
    \re^{(t-s)A}
    \cov(X(s)),\\
\nonumber
    \cov(X(s))
    & :=
    \cov(X(s),X(s))\\
\label{Xcovs}
    & =
    \re^{sA}\cov(X(0))\re^{sA^{\rT}} + \Sigma(s) + i\Theta.
\end{align}
The relations (\ref{Xcovts}) and (\ref{Xcovs}) follow from the decomposition (\ref{XY}) and the fact that the entries of $X(s)$ commute with and are statistically independent of (and hence, uncorrelated with) those of $Y(s,t)$ in (\ref{Y}).

While the mean square performance criteria in the quantum LQG control and filtering problems \cite{EB_2005,NJP_2009,VP_2013a,VP_2013b} are based on one-point second-order moments of the system variables,
the multi-point Gaussian states involve the covariance functions (\ref{Xcovts}). These two-point moments are also involved in the risk-sensitive \cite{DDJW_2006,J_2004,J_2005} cost functionals which penalize higher-order moments of integral quantities.

\section{Quadratic-exponential functional}\label{sec:QEF}

Consider an adapted  quantum process $\varphi$ defined for any time $t\> 0$ as the integral of a  quadratic function $\psi$  of the system variables over the time interval $[0,t]$:
\begin{equation}
\label{phi_psi}
    \varphi(t)
    :=
    \int_0^t
    \psi(s)
    \rd s,
    \qquad
    \psi(s)
    :=
    X(s)^{\rT} \Pi X(s).
\end{equation}
Here, $\Pi$ is a given real symmetric matrix of order $n$, so that $\varphi(t)$ and $\psi(t)$ are self-adjoint operators on the system-field space $\fH$ in (\ref{fH}). Furthermore, if $\Pi\succcurlyeq 0$, then
\begin{equation}
\label{psiZ}
    \psi
    =
    Z^{\rT}Z
    =
    \sum_{k=1}^n
    Z_k^2,
    \qquad
    Z:=
    \begin{bmatrix}
        Z_1\\
        \vdots\\
        Z_n
    \end{bmatrix}
    :=
    \sqrt{\Pi} X,
\end{equation}
and both $\varphi(t)$ and $\psi(t)$ are positive semi-definite. Now, consider the
exponential moment of the operator $\varphi(t)$   from (\ref{phi_psi}) given by
\begin{equation}
\label{QEF}
    \Xi_{\theta}(t)
    :=
    \bE \re^{\theta\varphi(t)}
    =
    \bE\re^{\int_0^t X(s)^{\rT} \Pi X(s)\rd s},
\end{equation}
where $\theta$ is a nonnegative real-valued  parameter (which is assumed to be sufficiently small in order for $\Xi_{\theta}(t)$ to be finite).  The dependence of $\Xi_{\theta}(t)$ on the matrix $\Pi$ through (\ref{phi_psi}) is omitted for brevity.   The quadratic-exponential functional (QEF) in (\ref{QEF}) can be regarded as an alternative to the quantum risk-sensitive performance  criterion in \cite{DDJW_2006,J_2004,J_2005}. The latter was defined as a weighted mean square value of a time-ordered exponential satisfying an operator differential equation (see \cite[Eqs. (19)--(21)]{J_2005}). In fact, the QEF is a straightforward quantum version of its classical predecessors \cite{BV_1985,J_1973,W_1981}.
For any given time $t\> 0$, the quantity $\Xi_{\theta}(t)$ is the moment-generating function for the quantum variable $\varphi(t)$ in the sense that
\begin{equation}
\label{Ephir}
    \bE(\varphi(t)^k)
    =
    \d_{\theta}^k \Xi_{\theta}(t)\big|_{\theta =0}
\end{equation}
for any positive integer $k$ for which the moment exists.  In particular, if the system variables have finite second moments (\ref{finsec}), then the asymptotic behaviour of the QEF for small values of $\theta$ is described by
\begin{equation}
\label{Tay}
    \Xi_{\theta}(t) = 1 + \theta \int_0^t \bra \Pi,  P(s)\ket\rd s + o(\theta),
    \qquad
    \theta \to 0,
\end{equation}
where $\bra K, N\ket:= \Tr(K^*N)$ denotes the Frobenius inner product \cite{HJ_2007} of real or complex matrices $K$ and $N$. Here,
\begin{equation}
\label{Pt}
    P(t):= \Re \bE\big(X(t)X(t)^{\rT}\big)
\end{equation}
is the real part of the matrix $\bE(XX^{\rT}) = P + i\Theta$ of second moments of the system variables governed by the Lyapunov ODE
\begin{equation}
\label{Pdot}
    \dot{P} = AP + PA^{\rT} + BB^{\rT}
\end{equation}
whose initial condition satisfies the generalized Heisenberg uncertainty principle $P(0)+i\Theta \succcurlyeq 0$ (see, for example, \cite{H_2001}).
Moreover, the first two terms on the right-hand side of (\ref{Tay}) provide a lower bound for the QEF. Indeed, for any fixed but otherwise arbitrary time $t\>0$, the self-adjoint operator $\varphi(t)$ can be regarded as a classical real-valued random variable with a probability distribution $E_t$.\footnote{This distribution
is related by $E_t(A):= \bE \sP_t(A)$ to the spectral measure $\sP_t$ of $\varphi(t)$, which is a projection-valued measure on the $\sigma$-algebra $\fB$ of Borel subsets of the real line satisfying $\sP_t(A)\sP_t(B) = \sP_t(A\bigcap B)$ for all $A,B\in \fB$ and the resolution of the identity property  $\sP_t(\mR) = \cI$; see, for example, \cite{H_2001}.} Hence,
\begin{align}
\nonumber
    \Xi_{\theta}(t)
    & =
    \int_{-\infty}^{+\infty} \re^{\theta x} E_t(\rd  x)\\
\nonumber
    & \>
    \int_{-\infty}^{+\infty} (1+\theta x) E_t(\rd x)
    =
    1 + \theta \bE \varphi(t)\\
\label{rex}
    & =
    1 + \theta \int_0^t \bE \psi(s)\rd s
    =
    1 + \theta\int_0^t \bra \Pi, P(s)\ket\rd s,
\end{align}
where use is also made of the inequality $\re^v \> 1+v$ on the real line, and $P$ is the function given by  (\ref{Pt}) and (\ref{Pdot}).    Furthermore, in combination with the first equality in (\ref{rex}) and convexity of the exponential function, the Jensen inequality leads to
\begin{align}
\nonumber
    \ln \Xi_{\theta}(t)
    & =
    \ln \int_{-\infty}^{+\infty} \re^{\theta x} E_t(\rd  x)\\
\nonumber
    & \>
    \theta \int_{-\infty}^{+\infty} x E_t(\rd x)
    =
    \theta \bE \varphi(t)\\
\label{Jen}
    & =
    \theta \int_0^t \bE \psi(s)\rd s
    =
    \theta\int_0^t \bra \Pi, P(s)\ket\rd s.
\end{align}
In view of (\ref{rex}) and (\ref{Jen}), the QEF $\Xi_{\theta}(t)$ in (\ref{QEF}) and its logarithm provide upper bounds for the mean square cost functional which is used in the quantum LQG control and filtering problems \cite{EB_2005,JNP_2008,NJP_2009}. An appropriate scaling represents the QEF  as a convex combination of the moments from (\ref{Ephir}):
\begin{equation}
\label{Ximom}
    \re^{-\theta}
    \Xi_{\theta}(t)
    =
    \re^{-\theta}
    \sum_{k=0}^{+\infty}
    \frac{\theta^k}{k!}
    \bE(\varphi(t)^k).
\end{equation}
Its coefficients $\frac{\theta^k}{k!}\re^{-\theta}$ constitute the Poisson probability mass function \cite{S_1996}  
with intensity parameter $\theta\>0$  and achieve their maximum at $k= \lfloor\theta\rfloor$, where $\lfloor \cdot\rfloor$ is the floor function. Therefore, $\theta$ determines the order of the moment of $\varphi(t)$, which is endowed with the largest weight in (\ref{Ximom}) and hence, is penalized most by the QEF as a cost functional. In comparison with the quadratic cost functionals in LQG control and filtering problems, the QEF  leads to minimizing not only the one-point second-order moments of the system variables but also their multi-point higher-order moments,  with the latter being subject to a stronger penalty for larger values of $\theta$.
A similar consideration applies to the cumulant-generating function
\begin{align}
\nonumber
    \ln \Xi_{\theta}(t)
    & =
    \sum_{r=1}^{+\infty}
    \frac{\theta^r}{r!}
    \bK_r(\varphi(t))\\
\label{expquad1}
    & =
    \theta
    \Big(
        \bE
        \varphi(t)
        +
        \frac{\theta}{2}
        \var(\varphi(t))
    \Big)
    +
    O(\theta^3),
    \qquad
    {\rm as}\
    \theta \to 0,
\end{align}
for the self-adjoint quantum variable $\varphi(t)$ as a classical random variable (with the distribution $E_t$ at a given moment of time $t$).
Here, the variance
$
    \var(\phi)
     :=
    \bE(\phi^2) - (\bE\phi)^2
    =
    \bK_2(\phi)
$
of a random variable $\phi$ is the second of the cumulants
\begin{equation}
\label{bK}
    \bK_r(\phi)
    :=
        \d_v^r
        \ln \bE \re^{v\phi}
    \big|_{v=0}
     =
    P_r(\bE \phi, \ldots, \bE (\phi^r)).
\end{equation}
For any $r\>1$,
the $r$th cumulant $\bK_r(\phi)$ is expressed in terms of the
first $r$ moments $\bE \phi, \ldots, \bE (\phi^r)$ through
the polynomial
\begin{align}
\nonumber
    P_r(\mu_1, \ldots, \mu_r)
    & =
    -r!
    \sum_{k=1}^r
    \frac{(-1)^k}{k}
    \sum_{j_1,\ldots, j_k\> 1:\ j_1+\ldots+j_k = r}\
    \prod_{s=1}^k
    \frac{\mu_{j_s}}{j_s!}\\
\label{bKP}
    & =
    \mu_r
    -r!
    \sum_{k=2}^r
    \frac{(-1)^k}{k}
    \sum_{j_1,\ldots, j_k\> 1:\ j_1+\ldots+j_k = r}\
    \prod_{s=1}^k
    \frac{\mu_{j_s}}{j_s!},
\end{align}
whose coefficients do not depend on the probability distribution of $\phi$. In particular, the first three polynomials are given by
\begin{align*}
    P_1(\mu_1)
    & =
    \mu_1,\\
    P_2(\mu_1, \mu_2)
    & =
    \mu_2 - \mu_1^2,\\
    P_3(\mu_1, \mu_2, \mu_3)
     & =
    \mu_3 - 3\mu_1\mu_2 + 2\mu_1^3.
\end{align*}
In view of (\ref{phi_psi}), the first two terms on the right-hand side of (\ref{expquad1}) involve the second and fourth moments of the system variables over the time interval $[0,t]$ and will be employed in Section~\ref{sec:quart} for approximating the QEF.

As mentioned above, the QEF (\ref{QEF}) differs from its predecessor in \cite{J_2004} and \cite[Eqs. (19)--(21)]{J_2005} which was used for measurement-based risk-sensitive  quantum control problems and defined in terms of an operator differential equation. Nevertheless, the QEF satisfies an integro-differential equation with somewhat similar features, which will be provided by Theorem~\ref{th:ODE} below.
For its formulation, we will need an auxiliary integral operator which is introduced by the following lemma.  To this end, for any time $t\>0$, let $\mho_t$ denote a linear map which acts as
\begin{align}
\nonumber
    \mho_t(\alpha, \beta)
    := &
        \Re
        \Big(
            X(t)^{\rT}
            \int_0^t
            \alpha(\sigma)
            X(\sigma)
            \rd \sigma
        \Big)\\
\label{mho}
        & +
        \int_{[0,t]^2}
        X(\sigma)^{\rT}
        \beta(\sigma,\tau)
        X(\tau)
        \rd \sigma
        \rd \tau
\end{align}
on a pair of locally integrable matrix-valued functions $\alpha: \mR_+\to \mR^{n\x n}$ and $\beta:\mR_+^2 \to \mR^{n\x n}$, the second of which satisfies
\begin{equation}
\label{betasym}
    \beta(\sigma,\tau)^{\rT} = \beta(\tau,\sigma)
\end{equation}
for all $\sigma,\tau\>0$. Here, the real part $\Re(\cdot)$ is extended to operators as $\Re \xi:= \frac{1}{2}(\xi+\xi^{\dagger})$. In view of (\ref{betasym}), the image $\mho_t(\alpha,\beta)$ in (\ref{mho}) is a self-adjoint quantum variable on $\fH$ which depends in a quadratic fashion on the past history of the system variables of the OQHO over the time interval $[0,t]$.

\begin{lemma}
\label{lem:intop}
For any given $t\>0$, the composition of the map $\mho_t$ in (\ref{mho}) with the commutator $\ad_{\varphi(t)}(\cdot):= [\varphi(t),\cdot]$, associated  with the process $\varphi$ in (\ref{phi_psi}), can be represented as
\begin{equation}
\label{admho}
    \ad_{\varphi(t)} \mho_t = i\mho_t \Gamma_t.
\end{equation}
Here, $\Gamma_t$ is a linear integral operator which maps the function pair $(\alpha, \beta)$ 
in (\ref{mho}) and (\ref{betasym}) to the pair $(\wt{\alpha}, \wt{\beta}) = \Gamma_t(\alpha,\beta)$ of such functions as
\begin{align}
\label{alphanew}
    \wt{\alpha}(\sigma)
    := &
    -4
    \int_0^t
    \alpha(s)\Lambda(s-\sigma)
    \rd s
    \Pi,\\
\nonumber
    \wt{\beta}(\sigma,\tau)
    := &
     2(\Pi \Lambda(\sigma-t)\alpha(\tau)-\alpha(\sigma)^{\rT}\Lambda(t-\tau)\Pi)\\
\label{betanew}
    & + 4
    \int_0^t
    \big(
        \Pi \Lambda(\sigma-s)\beta(s,\tau)
        -
        \beta(\sigma,s)\Lambda(s-\tau) \Pi
    \big)
    \rd s,
\end{align}
where $\Lambda$ is the two-point commutator function of the system variables from (\ref{Lambda}) and (\ref{XsXtcomm}).\hfill$\square$
\end{lemma}
\begin{proof}
Note that
\begin{align}
\nonumber
    [\xi, \Re\eta]
    & =
    \frac{1}{2}
    ([\xi, \eta] + [\xi,\eta^{\dagger}])
    =
    \frac{1}{2}
    ([\xi, \eta] + [\eta,\xi^{\dagger}]^{\dagger})\\
\label{xiReeta}
    & =
    \frac{1}{2}
    ([\xi, \eta] - [\xi,\eta]^{\dagger})= i \Im[\xi,\eta]
\end{align}
for any self-adjoint operator $\xi$ and an arbitrary operator $\eta$, where $\Im \zeta := \frac{1}{2i}(\zeta -\zeta^{\dagger})$ is the operator extension of the imaginary part $\Im(\cdot)$.
 A combination of (\ref{phi_psi}),  (\ref{mho}) and (\ref{xiReeta}) implies that
\begin{align}
\nonumber
    [\varphi(t), \mho_t(\alpha, \beta)]
    = &
    \int_0^t
    [\psi(s), \mho_t(\alpha, \beta)]
    \rd s\\
\nonumber
    = &
    \int_0^t
    \Big[
        \psi(s),
        \Re
        \Big(
        X(t)^{\rT}
        \int_0^t
        \alpha(\sigma)
        X(\sigma)
        \rd \sigma
        \Big)
    \Big]
    \rd s\\
\nonumber
    & +
    \int_0^t
    \Big[
        \psi(s),
        \int_{[0,t]^2}
        X(\sigma)^{\rT}
        \beta(\sigma,\tau)
        X(\tau)
        \rd \sigma
        \rd \tau
    \Big]
    \rd s\\
\nonumber
    = &
    i
    \Im
    \int_{[0,t]^2}
    [
        \psi(s),
        X(t)^{\rT}
        \alpha(\sigma)
        X(\sigma)
    ]
    \rd s
    \rd \sigma\\
\label{phimho}
    & +
    \int_{[0,t]^3}
    [
        \psi(s),
        X(\sigma)^{\rT}
        \beta(\sigma,\tau)
        X(\tau)
    ]
    \rd s
    \rd \sigma
    \rd \tau.
\end{align}
Now, the quadratic polynomials $\sum_{j=1}^r \sum_{k=1}^s \gamma_{jk}X_j(\sigma_j)X_k(\tau_k)$ of the system variables of the OQHO, considered at arbitrary moments of time $\sigma_1, \ldots, \sigma_r$ and $\tau_1,\ldots, \tau_s$, with complex coefficients $\gamma_{jk}$ (or integrals of such polynomials), form a Lie algebra with respect to the commutator; see Appendix~\ref{sec:quadcom} and references therein.
Furthermore, the two-point CCRs (\ref{XsXtcomm}) allow (\ref{quadcom2}) of Lemma~\ref{lem:quadcom} to be applied to the quadratic functions of the system variables on the right-hand side of (\ref{phimho}) as
\begin{align}
\nonumber
        [
        \psi(s),
        X(t)^{\rT}
        \alpha(\sigma)
        X(\sigma)
    ]
    & =
        [
        X(s)^{\rT}\Pi X(s),
        X(t)^{\rT}
        \alpha(\sigma)
        X(\sigma)
    ]\\
\nonumber
   & =
    4i
    \begin{bmatrix}
        X(s) \\
        X(t)
    \end{bmatrix}^{\rT}
    \begin{bmatrix}
      0 & \Pi \Lambda(s-t)\alpha(\sigma)\\
      \alpha(\sigma)\Lambda(s-\sigma)^{\rT} \Pi & 0
    \end{bmatrix}
    \begin{bmatrix}
        X(s) \\
        X(\sigma)
    \end{bmatrix}\\
\label{comm1}
    & =
    4i
    \big(
        X(s)^{\rT}\Pi \Lambda(s-t)\alpha(\sigma)X(\sigma)
        -
        X(t)^{\rT} \alpha(\sigma)\Lambda(\sigma-s)\Pi X(s)
    \big),\\
\nonumber
    [
        \psi(s),
        X(\sigma)^{\rT}
        \beta(\sigma,\tau)
        X(\tau)
    ]
    & =
    [
        X(s)^{\rT}\Pi X(s),
        X(\sigma)^{\rT}
        \beta(\sigma,\tau)
        X(\tau)
    ]\\
\nonumber
    & =
    4i
    \begin{bmatrix}
        X(s) \\
        X(\sigma)
    \end{bmatrix}^{\rT}
    \begin{bmatrix}
      0 & \Pi \Lambda(s-\sigma)\beta(\sigma,\tau)\\
      \beta(\sigma,\tau)\Lambda(s-\tau)^{\rT} \Pi & 0
    \end{bmatrix}
    \begin{bmatrix}
        X(s) \\
        X(\tau)
    \end{bmatrix}\\
\label{comm2}
    & =
    4i
    \big(
        X(s)^{\rT}\Pi \Lambda(s-\sigma)\beta(\sigma,\tau)X(\tau)
        -
        X(\sigma)^{\rT}\beta(\sigma,\tau)\Lambda(\tau-s) \Pi X(s)
    \big),
\end{align}
where the first equality from (\ref{Lambda}) has also been taken into account.  Substitution of (\ref{comm1}) and (\ref{comm2}) into (\ref{phimho}) leads to
\begin{align}
\nonumber
    [\varphi(t), \mho_t(\alpha, \beta)]
    = &
    4i
    \Re
    \int_{[0,t]^2}
    \big(
        X(s)^{\rT}\Pi \Lambda(s-t)\alpha(\sigma)X(\sigma)
        -
        X(t)^{\rT} \alpha(\sigma)\Lambda(\sigma-s)\Pi X(s)
    \big)
    \rd s
    \rd \sigma\\
\nonumber
    & +
    4i
    \int_{[0,t]^3}
    \big(
        X(s)^{\rT}\Pi \Lambda(s-\sigma)\beta(\sigma,\tau)X(\tau)
        -
        X(\sigma)^{\rT}\beta(\sigma,\tau)\Lambda(\tau-s) \Pi X(s)
    \big)
    \rd s
    \rd \sigma
    \rd \tau\\
\nonumber
    = &
    -4i
    \Re
    \Big(
    X(t)^{\rT}
    \int_{[0,t]^2}
        \alpha(s)\Lambda(s-\sigma)\Pi X(\sigma)
    \rd s
    \rd \sigma
        \Big)\\
\nonumber
    & +
    4i
    \Re
    \int_{[0,t]^2}
        X(\sigma)^{\rT}\Pi \Lambda(\sigma-t)\alpha(\tau)X(\tau)
    \rd \sigma
    \rd \tau\\
\label{phimho1}
    & +
    4i
    \int_{[0,t]^3}
    X(\sigma)^{\rT}
    \big(
        \Pi \Lambda(\sigma-s)\beta(s,\tau)
        -
        \beta(\sigma,s)\Lambda(s-\tau) \Pi
    \big)
    X(\tau)
    \rd s
    \rd \sigma
    \rd \tau,
\end{align}
where use is also made of appropriate permutations of the integration variables. By grouping the terms on the right-hand side of (\ref{phimho1}) and comparing the result with (\ref{mho}), it follows that
\begin{equation}
\label{phimho2}
    [\varphi(t), \mho_t(\alpha, \beta)]
    =
    i\mho_t(\wt{\alpha}, \wt{\beta})
    =
    i\mho_t (\Gamma_t(\alpha,\beta)),
\end{equation}
where the functions $\wt{\alpha}$ and $\wt{\beta}$ are related to $\alpha$ and $\beta$ through the linear operator $\Gamma_t$ specified by (\ref{alphanew}) and (\ref{betanew}). The function $\wt{\beta}$ in (\ref{betanew}) inherits the property (\ref{betasym}) due to the symmetry of the matrix $\Pi$ and the first equality in (\ref{Lambda}). The fulfillment of (\ref{phimho2}) for arbitrary functions $\alpha$ and $\beta$, described above, establishes (\ref{admho}).
\end{proof}

Lemma~\ref{lem:intop} shows that the class of quantum variables $\mho_t(\alpha,\beta)$, described by (\ref{mho}) for admissible functions $\alpha$ and $\beta$ is invariant under the action of $i\ad_{\varphi(t)}$. Moreover, (\ref{admho}) extends as
\begin{equation}
\label{fadmho}
    f(i\ad_{\varphi(t)}) \mho_t = \mho_t f(-\Gamma_t)
\end{equation}
to entire functions $f$ of a complex variable evaluated at operators. The nature of the operator $f(-\Gamma_t)$ in (\ref{fadmho}) is simpler than that of $f(i\ad_{\varphi(t)})$. Indeed, $\Gamma_t$ (for any fixed but otherwise arbitrary $t\>0$) in (\ref{alphanew}), (\ref{betanew}) is a bounded operator on the Banach space of pairs of integrable $\mR^{n\x n}$-valued functions on the time interval $[0,t]$, described above, and hence, so also is the operator $f(-\Gamma_t)$. This boundedness follows from the uniform boundedness of the function $\Lambda$ in (\ref{Lambda}) due to the matrix $A$ being Hurwitz.

Also, for the purposes of the following theorem, we will need a time-varying change of the density operator in (\ref{rho}):
\begin{equation}
\label{rhot}
    \rho_{\theta, t}
    :=
    \frac{1}{\Xi_{\theta}(t)}
    \re^{\frac{\theta}{2} \varphi(t)} \rho \re^{\frac{\theta}{2} \varphi(t)}.
\end{equation}
The property that $\rho_{\theta, t}$ is a density operator is ensured by its self-adjointness, positive semi-definiteness and the unit trace
\begin{align*}
    \Tr \rho_{\theta, t}
     & =
    \frac{1}{\Xi_{\theta}(t)}
    \Tr \Big(\re^{\frac{\theta}{2} \varphi(t)} \rho \re^{\frac{\theta}{2} \varphi(t)}\Big)\\
     & =
    \frac{1}{\Xi_{\theta}(t)}
    \Tr \big(\rho \re^{\theta \varphi(t)}\big)
    =
    \frac{1}{\Xi_{\theta}(t)}
    \bE \re^{\theta \varphi(t)}=1
\end{align*}
in view of (\ref{QEF}).
The expectation over the modified density operator $\rho_{\theta, t}$ in (\ref{rhot}) is denoted by $\bE_{\theta, t}$ and computed as
\begin{align}
\nonumber
  \bE_{\theta, t}\xi
  & :=
  \Tr(\rho_{\theta,t}\xi)\\
\nonumber
  & =
    \frac{1}{\Xi_{\theta}(t)}
    \Tr
    \big(
    \re^{\frac{\theta}{2} \varphi(t)} \rho \re^{\frac{\theta}{2} \varphi(t)}
    \xi
    \big)\\
\label{bEt}
    & =
    \frac{1}{\Xi_{\theta}(t)}
    \bE
    \big(
    \re^{\frac{\theta}{2} \varphi(t)}
    \xi
    \re^{\frac{\theta}{2} \varphi(t)}
    \big).
\end{align}
In the case $\theta=0$, the  definitions (\ref{rhot}) and (\ref{bEt}) reproduce the original density operator $\rho_{0,t}=\rho$ and the original expectation $\bE_{0,t}=\bE$.

\begin{theorem}
\label{th:ODE}
As a function of time, the QEF $\Xi_{\theta}(t)$ in (\ref{QEF}), associated with the OQHO (\ref{dX}), satisfies the integro-differential equation
\begin{equation}
\label{dotQEF}
    \d_t \ln \Xi_{\theta}(t)
    =
    \theta
    \bE_{\theta,t}
    \big(
        \psi(t) +
        \Ups_{\theta}(t)
    \big),
\end{equation}
where $\bE_{\theta,t}$ is the modified expectation (\ref{bEt}), and  the process $\psi$ is given by (\ref{phi_psi}). Here,
$\Ups_{\theta}(t)$ is a time-varying self-adjoint operator which depends in a quadratic fashion on the past history of the system variables according to (\ref{mho}) as
\begin{align}
\nonumber
    \Ups_{\theta}(t)
    := &
    \frac{\theta}{2}
    \mho_t(\alpha_{\theta,t}, \beta_{\theta,t})
    \\
\nonumber
    = &
    \frac{\theta}{2}
    \left(
        \Re
        \Big(
            X(t)^{\rT}
            \int_0^t
            \alpha_{\theta,t}(\sigma)
            X(\sigma)
            \rd \sigma
        \Big)
    \right.\\
\label{Ups0}
        &+
    \left.
        \int_{[0,t]^2}
        X(\sigma)^{\rT}
        \beta_{\theta,t}(\sigma,\tau)
        X(\tau)
        \rd \sigma
        \rd \tau
    \right).
\end{align}
Also, the pair of $\mR^{n\x n}$-valued functions $\alpha_{\theta, t}$ and $\beta_{\theta,t}$ (the second of which satisfies the symmetry property (\ref{betasym})) is computed as the image
\begin{equation}
\label{kerfun}
    (\alpha_{\theta, t}, \beta_{\theta,t})
    =
    f\Big(\frac{\theta}{2}\Gamma_t\Big)(\wh{\alpha}_t,0)
\end{equation}
of the pair $(\wh{\alpha}_t,0)$ under the linear operator $f\big(\frac{\theta}{2}\Gamma_t\big)$, where the integral operator $\Gamma_t$ is described in Lemma~\ref{lem:intop}. Also, the function $\wh{\alpha}_t$ is given by
\begin{equation}
\label{alphahat}
  \wh{\alpha}_t(\sigma) := -8\Pi \Lambda(t-\sigma) \Pi
\end{equation}
in terms of the two-point commutator function $\Lambda$ of the system variables from (\ref{Lambda}) and (\ref{XsXtcomm}). Furthermore,
\begin{equation}
\label{f}
    f(z)
    :=
    \left\{
    \begin{matrix}
        0 & {\rm if}\ z= 0\\
        \frac{\sin z-z}{z^2} & {\rm otherwise}\end{matrix}
    \right.
    =
    \sum_{k=1}^{+\infty}
    (-1)^k
    \frac{z^{2k-1}}{(2k+1)!}
\end{equation}
is an entire odd function of a complex variable which is evaluated in (\ref{kerfun}) at a bounded  operator $\frac{\theta}{2}\Gamma_t$.
\hfill$\square$
\end{theorem}
\begin{proof}
Application of the Magnus lemma on the differentiation of operator exponentials \cite{H_1996,K_1976,L_1968,M_1954,W_1967}  leads to
\begin{align}
\nonumber
    (\re^{\theta \varphi(t)})^{^\centerdot}
    & =
    \theta
    \int_0^1
    \re^{\lambda\theta \varphi(t)}
    \dot{\varphi}(t)
    \re^{(1-\lambda)\theta \varphi(t)}
    \rd \lambda\\
\nonumber
    & =
    \theta
    \re^{\frac{\theta}{2}\varphi(t)}
    \int_{-\frac{1}{2}}^{\frac{1}{2}}
    \re^{\lambda\theta \varphi(t)}
    \psi(t)
    \re^{-\lambda\theta \varphi(t)}
    \rd \lambda
    \re^{\frac{\theta}{2}\varphi(t)}
    \\
\nonumber
    & =
    \theta
    \re^{\frac{\theta}{2}\varphi(t)}
    \int_{-\frac{1}{2}}^{\frac{1}{2}}
    \re^{\lambda\theta \ad_{\varphi(t)}}
    (\psi(t))
    \rd \lambda
    \re^{\frac{\theta}{2}\varphi(t)}\\
\label{dotephi}
    & =
    \theta
    \re^{\frac{\theta}{2}\varphi(t)}
    \sinhc
    \Big(
        \frac{\theta}{2}
        \ad_{\varphi(t)}
    \Big)
    (\psi(t))
    \re^{\frac{\theta}{2}\varphi(t)}.
\end{align}
Here,
$
    \dot{\varphi}(t) = \psi(t)
$ is the time derivative
of the process $\varphi$ in (\ref{phi_psi}), and
$$
    \re^{\mu\varphi}
    \psi
    \re^{-\mu\varphi}
    =
    \re^{\mu \ad_{\varphi}}(\psi)
     =
    \sum_{k=0}^{+\infty}
    \frac{\mu^k}{k!}
    \ad_{\varphi}^k(\psi), 
$$  
where
$
    \ad_{\varphi}^k(\cdot):= \overbrace{[\varphi, [\varphi, \ldots [\varphi,}^{k\ {\rm times}} \,\cdot\,]\ldots]
$
is the $k$-fold iterate of the commutator $\ad_{\varphi(t)}(\cdot)$ (with $\ad_{\varphi}^0$ being the identity map by the standard convention).
Also,
\begin{align}
\nonumber
    \sinhc(z)
    &:=
    \frac{1}{2}
    \int_{-1}^1
    \re^{\lambda z}
    \rd \lambda
    =
    \left\{
    \begin{matrix}1 & {\rm if}\ z= 0\\
    \frac{\sinh z}{z} & {\rm otherwise}\end{matrix}
    \right.\\
\label{sinhc}
    & =
    \sum_{k=0}^{+\infty}
    \frac{z^{2k}}{(2k+1)!}
\end{align}
is a hyperbolic version of the sinc function (in the sense that $\sinhc(iz) = \sinc(z)$).
This is an entire even function of a complex variable $z$, which is evaluated in (\ref{dotephi}) at the superoperator $        \frac{\theta}{2}
        \ad_{\varphi}
$. Since $\sinhc (0)=1$ in view of (\ref{sinhc}), then
\begin{equation}
\label{sinhcad}
    \sinhc
    \Big(
        \frac{\theta}{2}
        \ad_{\varphi(t)}
    \Big)
    (\psi(t))
    =
    \psi(t) + \Ups_{\theta}(t),
\end{equation}
where
\begin{align}
\nonumber
    \Ups_{\theta}(t)
    & :=
    \sum_{k=1}^{+\infty}
    \frac{(\theta/2)^{2k}}{(2k+1)!}
    \ad_{\varphi(t)}^{2k}(\psi(t))\\
\label{Ups}
    & =
    \frac{\theta}{2}
    g
    \Big(
        \frac{\theta}{2}\ad_{\varphi(t)}
    \Big)([\varphi(t), \psi(t)]).
\end{align}
Here,
\begin{equation}
\label{g}
    g(z)
    :=
    \left\{
    \begin{matrix}
        0 & {\rm if}\ z= 0\\
        \frac{\sinhc(z)-1}{z} & {\rm otherwise}\end{matrix}
    \right.
    =
    \sum_{k=1}^{+\infty}
    \frac{z^{2k-1}}{(2k+1)!}
\end{equation}
is an entire odd function on the complex plane, which is applied to the superoperator  $\frac{\theta}{2}\ad_{\varphi(t)}$ in (\ref{Ups}), with the resulting superoperator acting on the (skew-Hermitian) quantum variable $[\varphi(t), \psi(t)]$. The self-adjointness of the operator $\Ups_{\theta}(t)$ follows directly from (\ref{Ups}). Indeed, since $\ad_{i\xi}(\eta) = i[\xi,\eta]$ is a self-adjoint operator for any self-adjoint operators $\xi$ and $\eta$, then so also is $\ad_{\xi}^2(\eta) = - \ad_{i\xi}^2(\eta)$, and, by induction, $\ad_{\xi}^{2k}(\eta)$ is self-adjoint for all  $k=1,2,3,\ldots$. Therefore, the series (\ref{Ups}) consists of self-adjoint operators due to the self-adjointness of $\varphi(t)$ and $\psi(t)$ in (\ref{phi_psi}).  Now, substitution of (\ref{sinhcad}) into (\ref{dotephi}) yields
\begin{equation}
\label{dotephi1}
    (\re^{\theta \varphi(t)})^{^\centerdot}
    =
    \theta
    \re^{\frac{\theta}{2}\varphi(t)}
    (\psi(t) + \Ups_{\theta}(t))
    \re^{\frac{\theta}{2}\varphi(t)}.
\end{equation}
By taking the expectation on both sides of (\ref{dotephi1}) and using (\ref{QEF}) together with the modified expectation from (\ref{bEt}), it follows that
\begin{align}
\nonumber
    \d_t \Xi_{\theta}(t)
    & =
    \bE
    \big(
        (\re^{\theta \varphi(t)})^{^\centerdot}
    \big)\\
\nonumber
    & =
    \theta
    \bE
    \Big(\re^{\frac{\theta}{2}\varphi(t)}
    (\psi(t) + \Ups_{\theta}(t))
    \re^{\frac{\theta}{2}\varphi(t)}    \Big)\\
\label{dotQEF1}
    & =
    \theta
    \Xi_{\theta}(t)
    \bE_{\theta,t}
    \big(
    \psi(t) + \Ups_{\theta}(t)
    \big).
\end{align}
Division of both parts of (\ref{dotQEF1}) by the QEF $\Xi_{\theta}(t)$ leads to its logarithmic derivative in (\ref{dotQEF}). 
We will now compute the process $\Ups_{\theta}$ in (\ref{Ups}). To this end,
use will be made of the Lie-algebraic property of quadratic polynomials of system variables of the OQHO, which was already employed in the proof of Lemma~\ref{lem:intop}. 
More precisely, in view of the two-point CCRs (\ref{XsXtcomm}), application of (\ref{quadcom2}) from Lemma~\ref{lem:quadcom} to the quadratic functions of the system variables in (\ref{phi_psi}) leads to
\begin{align}
\nonumber
    [\psi(s),\psi(t)]
    & =
    [X(s)^{\rT}\Pi X(s), X(t)^{\rT}\Pi X(t)]\\
\nonumber
    & =
    4i
    \begin{bmatrix}
        X(s) \\
        X(t)
    \end{bmatrix}^{\rT}
    \begin{bmatrix}
      0 & \gamma(s-t)\\
      -\gamma(t-s) & 0
    \end{bmatrix}
    \begin{bmatrix}
        X(s) \\
        X(t)
    \end{bmatrix}\\
\nonumber
    &=
    4i
    \big(
        X(s)^{\rT} \gamma(s-t) X(t)
        -
        X(t)^{\rT} \gamma(t-s) X(s)
    \big)\\
\label{psispsit}
    & =
    -8i
    \Re
    \big(
        X(t)^{\rT} \gamma(t-s) X(s)
    \big)
\end{align}
for any $s,t\>0$.
Here,
\begin{equation}
\label{gamma}
    \gamma(\tau)
    :=
    \Pi \Lambda(\tau)\Pi
    =
    -\gamma(-\tau)^{\rT},
\end{equation}
and
use is also made of (\ref{Lambda}) in combination with the symmetry of the matrix $\Pi$ and the identity $(\xi^{\rT} N\eta)^{\dagger} = \eta N^{\rT} \xi$ which holds for vectors $\xi$ and $\eta$ of self-adjoint operators and an appropriately dimensioned real matrix $N$. In view of (\ref{phi_psi}), it follows from (\ref{psispsit}) 
that
\begin{align}
\nonumber
    [\varphi(t), \psi(t)]
    & =
    \int_0^t
    [\psi(s), \psi(t)]
    \rd s\\
\nonumber
       & =
       -
       8i
       \Re
       \Big(
            X(t)^{\rT}
            \int_0^t
            \gamma(t-s)
            X(s)
            \rd s
        \Big)\\
\label{phitpsit}
        & =
        i \mho_t(\wh{\alpha}_t,0),
\end{align}
where the last equality is obtained by comparison with (\ref{mho}), with the function $\wh{\alpha}_t$ given by (\ref{alphahat}) in view of (\ref{gamma}). Substitution of (\ref{phitpsit}) into (\ref{Ups}) leads to
\begin{align}
\nonumber
    \Ups_{\theta}(t)
    & =
    \frac{\theta}{2}
    g
    \Big(
        \frac{\theta}{2}\ad_{\varphi(t)}
    \Big)(i \mho_t(\wh{\alpha}_t,0))\\
\nonumber
    & =
    i\frac{\theta}{2}
    \mho_t
    \Big(
    g
    \Big(
        i\frac{\theta}{2}\Gamma_t
    \Big)
    (\wh{\alpha}_t,0)
    \Big)\\
\label{Ups1}
    & =
    \frac{\theta}{2}
    \mho_t
    \Big(
    f
    \Big(
        \frac{\theta}{2}\Gamma_t
    \Big)
    (\wh{\alpha}_t,0)
    \Big).
\end{align}
Here, use is made of Lemma~\ref{lem:intop} (and its corollary (\ref{fadmho})), with the function $f$ in (\ref{f}) being related to $g$ in (\ref{g}) by
$
    f(z):= ig(iz)
$. In view of the notation (\ref{kerfun}), the representation (\ref{Ups1}) establishes (\ref{Ups0}), thus completing the proof.
\end{proof}

In the classical case (obtained in the limit as $\Theta\to 0$ and $J\to 0$), the operators $\varphi(t)$ and $\psi(t)$ in (\ref{phi_psi}) are real-valued (and hence, commuting) functions of time, and the process $\Ups_{\theta}$ vanishes. In this case,  (\ref{dotQEF}) reduces to
$$
    \d_t \ln \Xi_{\theta}(t)
    =
    \theta
    \bE_{\theta,t}
        \psi(t)
    =
    \frac{\theta}{\Xi_{\theta}(t)}
    \bE
    \Big(
        \re^{\theta \varphi(t)}
        \psi(t)
    \Big)
$$
and corresponds to the derivative of the usual exponential function. Returning to the quantum setting, we note that
(\ref{dotephi1}) resembles the operator differential equation \cite[Eq. (19)]{J_2005}. However, unlike the latter, (\ref{dotephi1}) produces a self-adjoint solution, and its right-hand side involves an additional process $\Ups_{\theta}$ which depends on the past history of the system variables. 
More precisely, the operator $\Ups_{\theta}(t)$ stores part of the quantum information on the past history of the system variables that is relevant to the time derivative $(\re^{\theta \varphi(t)})^{^\centerdot}$.

\section{A quartic approximation of the quadratic-exponential functional}\label{sec:quart}

As mentioned in the previous section in regard to the Taylor series expansion (\ref{expquad1}), the QEF can be approximated by the quantity
\begin{equation}
\label{F}
    F_{\theta}(t)
    :=
    \theta
    \Big(
        \bE
        \varphi(t)
        +
        \frac{\theta}{2}
        \var(\varphi(t))
    \Big)
\end{equation}
which takes into account only the first four moments of the system variables over the time interval $[0,t]$. In view of (\ref{rex}) and (\ref{Jen}), the term $\bE \varphi(t)$ in (\ref{F}) is organised as a mean square cost functional in the coherent quantum LQG control problems \cite{NJP_2009}. This quadratic term is revisited for completeness in the following theorem whose main result is the computation of the second term in (\ref{F}).

\begin{theorem}
\label{th:asy}
Suppose the matrix $A$ in (\ref{A_B}) is Hurwitz, and the OQHO (\ref{dX}), driven by vacuum fields, is initialised at the invariant Gaussian state.  Then the mean value and variance of the quantum process $\varphi$ in (\ref{phi_psi}) can be computed as
\begin{align}
\label{Ephi}
    \bE \varphi(t)
    & =
    \bra \Pi, P\ket t,\\
\label{Varphi}
    \var (\varphi(t))
    & =
    4
    \int_0^t
    (t-\tau)
    \Bra
        \Pi,
        \re^{\tau A}
        (P\Pi P + \Theta \Pi \Theta)
        \re^{\tau A^{\rT}}
    \Ket
    \rd \tau
\end{align}
for any $t\> 0$,
where the matrix $P$ is given by (\ref{P}). The infinite-horizon asymptotic behavior of this variance is described by
\begin{equation}
\label{Varphiasy}
    \lim_{t\to +\infty}
    \Big(
    \frac{1}{t}
    \var(\varphi(t))
    \Big)
    =
    4
    \bra
        \Pi,
        T
    \ket
    =
    4
    \bra
        Q,
        P\Pi P + \Theta \Pi \Theta
    \ket,
\end{equation}
where the matrices $T$ and $Q$
are the unique solutions of the ALEs
\begin{align}
\label{TALE}
    A T + TA^{\rT} + P\Pi P + \Theta \Pi \Theta & = 0,\\
\label{QALE}
    A^{\rT}Q + QA + \Pi & = 0.
\end{align}
 \hfill$\square$
\end{theorem}
\begin{proof}
Since the invariant Gaussian quantum state has zero mean and the covariance matrix $P+i\Theta$, the relation (\ref{Ephi}) follows from (\ref{phi_psi}) and  the equality
\begin{equation}
\label{Epsi}
    \bE\psi(t)
    =
    \Bra
        \Pi, \bE(X(t)X(t)^{\rT})
    \Ket
    =
    \bra \Pi, P\ket
\end{equation}
for any $t\>0$. Also, (\ref{phi_psi}) allows the variance of $\varphi(t)$ to be expressed in terms of the quantum covariance function of $\psi$ as
\begin{equation}
\label{Varphi1}
    \var(\varphi(t))
    =
    \int_{[0,t]^2}
    \cov(\psi(\sigma),\psi(\tau))
    \rd \sigma
    \rd \tau.
\end{equation}
Application of Lemma~\ref{lem:covquad} of Appendix~\ref{sec:covquad} to 
the two-point Gaussian quantum state for the system variables at the moments of time $\sigma$ and $\tau$ leads to
\begin{align}
\nonumber
  \cov(\psi(\sigma), \psi(\tau))
  & =
  \cov(
    X(\sigma)^{\rT} \Pi X(\sigma),
    X(\tau)^{\rT} \Pi X(\tau)
  )\\
\label{covpsi}
    & =
    2
        \Bra
            \Pi,
            S(\sigma-\tau)\Pi S(\sigma-\tau)^{\rT}
        \Ket,
\end{align}
where use is also made of the symmetry of the matrix $\Pi$.
Here, in accordance with (\ref{V})--(\ref{Xcovs}), the corresponding quantum covariance matrix of the system variables
\begin{align}
\nonumber
    \bE(X(\sigma)X(\tau)^{\rT})
    & =
    \left\{
        \begin{matrix}
            \re^{(\sigma-\tau)A}(P+i\Theta)
            & {\rm if} & \sigma \> \tau\> 0\\
            (P+i\Theta)
            \re^{(\tau-\sigma)A^{\rT}}
            & {\rm if} & \tau\> \sigma \> 0
        \end{matrix}
    \right.\\
\label{S}
    & =
    V(\sigma-\tau) + i\Lambda(\sigma-\tau)
    =
    S(\sigma-\tau)
\end{align}
depends on the difference $\sigma-\tau$ since the system is assumed to be initialised at the invariant state. 
A combination of (\ref{covpsi}) with an appropriate transformation of the integration variables in (\ref{Varphi1})  leads to
\begin{align}
\nonumber
    \var(\varphi(t))
    & =
    2
    \int_{[0,t]^2}
        \Bra
            \Pi,
            S(\sigma-\tau)\Pi S(\sigma-\tau)^{\rT}
        \Ket
    \rd \sigma
    \rd \tau\\
\nonumber
    &=
    2
    \int_{-t}^t
    (t-|\tau|)
    \Bra
        \Pi,
        S(\tau)\Pi S(\tau)^{\rT}
    \Ket
    \rd \tau\\
\nonumber
    &=
    2
    \int_0^t
    (t-\tau)
    \bra
        \Pi,
        S(\tau)\Pi S(\tau)^{\rT}
        +
        S(-\tau)\Pi S(-\tau)^{\rT}
    \ket
    \rd \tau\\
\label{Varphi2}
    &=
    4
    \int_0^t
    (t-\tau)
    \Bra
        \Pi,
        \Re
        \big(
            S(\tau)\Pi S(\tau)^{\rT}
        \big)
    \Ket
    \rd \tau.
\end{align}
Here, the symmetry of the matrix $\Pi$ has been  used together with the property $S(-\tau) = S(\tau)^*$ (see also (\ref{V})--(\ref{Lambda})), whereby
\begin{align*}
    \Bra
        \Pi,\,
        S(-\tau)\Pi S(-\tau)^{\rT}
    \Ket
    & =
    \Bra
        \Pi,\,
        S(\tau)^*\Pi \overline{S(\tau)}
    \Ket
    =
    \Bra
        S(\tau)\Pi,\,
        \Pi \overline{S(\tau)}
    \Ket\\
    & =
    \overline{
    \Bra
        \Pi \overline{S(\tau)},\,
        S(\tau)\Pi
    \Ket}
    =
    \overline{
    \Bra
        \Pi ,\,
        S(\tau)\Pi S(\tau)^{\rT}
    \Ket}
\end{align*}
for any $\tau\in \mR$.
By substituting (\ref{S}) into (\ref{Varphi2}) and using the relation $(P+i\Theta)^{\rT} = P-i\Theta$ in view of the symmetry  of $P$ and antisymmetry of $\Theta$, it follows that
\begin{align*}
\nonumber
    \var(\varphi(t))
    & =
    4
    \int_0^t
    (t-\tau)
    \bra
        \Pi,
        \re^{\tau A}
        \Re
        \big(
            (P+i\Theta)
            \Pi
            (P-i\Theta)
        \big)
        \re^{\tau A^{\rT}}
    \ket
    \rd \tau\\
    & =
    4
    \int_0^t
    (t-\tau)
    \Bra
        \Pi,
        \re^{\tau A}
        (P\Pi P+\Theta\Pi \Theta)
        \re^{\tau A^{\rT}}
    \Ket
    \rd \tau,
\end{align*}
which establishes (\ref{Varphi}). It now remains to prove (\ref{Varphiasy}). To this end, the representation (\ref{Varphi2}) implies that, for any $t>0$,
\begin{equation}
\label{Varphi4}
    \frac{1}{t}
    \var(\varphi(t))
    =
    4
    \int_0^{+\infty}
    \chi_t(\tau)
    \Bra
        \Pi,
        \Re
        \big(
            S(\tau)\Pi S(\tau)^{\rT}
        \big)
    \Ket
    \rd \tau,
\end{equation}
where
\begin{equation}
\label{chi}
    \chi_t(\tau)
    :=
    \max
    \Big(
        0,
        1-\frac{\tau}{t}
    \Big)
\end{equation}
is a finite support function which is bounded by and converges to $1$ as $t\to +\infty$ for any given $\tau\> 0$.  Since (due to $A$ being Hurwitz) the covariance function $S(\tau)$ in (\ref{S}) is square integrable over $\tau$, application of Lebesgue's dominated convergence theorem to (\ref{Varphi4}) leads to
\begin{equation}
\label{PiT}
    \lim_{t\to +\infty}
    \Big(
        \frac{1}{t}
        \var(\varphi(t))
    \Big)
    =
    4
    \int_0^{+\infty}
    \Bra
        \Pi,
        \Re
        \big(
            S(\tau)\Pi S(\tau)^{\rT}
        \big)
    \Ket
    \rd \tau
    =
    4
    \Bra
        \Pi,
        T
    \Ket,
\end{equation}
thus proving the first equality in (\ref{Varphiasy}),
where the matrix
\begin{equation}
\label{T}
    T
    :=
    \int_0^{+\infty}
    \Re
    \big(
        S(\tau)\Pi S(\tau)^{\rT}
    \big)
    \rd \tau
    =
        \int_0^{+\infty}
        \re^{\tau A}
        (P\Pi P + \Theta \Pi \Theta)
        \re^{\tau A^{\rT}}
        \rd \tau
\end{equation}
satisfies the ALE (\ref{TALE}). On the other hand, application of a duality argument allows the integral in (\ref{PiT}) to be evaluated as
\begin{align*}
    \int_0^{+\infty}
    \Bra
        \Pi,
        \Re
        \big(
            S(\tau)\Pi S(\tau)^{\rT}
        \big)
    \Ket
    \rd \tau
    & =
    \int_0^{+\infty}
    \Bra
        \Pi,
        \re^{\tau A}
        (P\Pi P+\Theta\Pi \Theta)
        \re^{\tau A^{\rT}}
    \Ket
    \rd \tau\\
    & =
    \int_0^{+\infty}
    \Bra
        \re^{\tau A^{\rT}}\Pi\re^{\tau A},
        P\Pi P+\Theta\Pi \Theta
    \Ket
    \rd \tau\\
    & =
    \Bra
        Q,
        P\Pi P+\Theta\Pi \Theta
    \Ket,
\end{align*}
which establishes the second of the equalities in (\ref{Varphiasy}), with the matrix
$  Q:=
    \int_0^{+\infty}
        \re^{\tau A^{\rT}}
        \Pi
        \re^{\tau A}
    \rd \tau
$
satisfying the ALE (\ref{QALE}).
\end{proof}

In the limiting case $\Theta= 0$, when the system variables commute with each other, the relation (\ref{covpsi}) reduces to the representation
\begin{equation}
\label{covquad}
    \cov(
        \xi^{\rT} \Pi\xi,
        \eta^{\rT} \Pi\eta
    )
    =
    2
    \Bra
        \Pi,
        \cov(\xi,\eta)
        \Pi
        \cov(\eta,\xi)
    \Ket
\end{equation}
for the covariance of quadratic forms of jointly Gaussian classical random vectors $\xi$ and $\eta$ with zero mean values (see, for example, \cite[Lemmas~2.3 and 6.2]{M_1978} and \cite[Lemma~6]{VP_2010b}). The noncommutative quantum nature of the setting under consideration with $\Theta \ne 0$ manifests itself in (\ref{Varphi})--(\ref{TALE})  through the term $\Theta \Pi\Theta$.

Also, in application to the quartic approximation (\ref{F}) of the QEF, Theorem~\ref{th:asy} yields the following infinite-horizon asymptotic behaviour:
\begin{equation}
\label{Fasy}
    \frac{1}{\theta}
    \lim_{t\to +\infty}
    \Big(
        \frac{1}{t}
        F_{\theta}(t)
    \Big)
    =
    \bra
        \Pi,
        P
        +
        2\theta T
    \ket.
\end{equation}
Here, $T$ is the matrix from (\ref{T}) which takes into account the two-point quantum  statistical correlations (\ref{S}) between the system variables. The relation (\ref{Fasy})  shows that, in the infinite-horizon limit, the quartic term $    \frac{\theta}{2}
        \var(\varphi(t))
$ in (\ref{F}) can be neglected in comparison with the quadratic term $
        \bE
        \varphi(t)
$ only if the parameter $\theta$ is small enough in the sense that
\begin{equation}
\label{theta0}
    \theta
    \ll
    \theta_0
    :=
    \frac{1}{2}
    \frac
    {\bra\Pi,P\ket}
    {\bra\Pi,T\ket}.
\end{equation}

\section{Asymptotic growth rates of the higher-order cumulants}\label{sec:cumul}

Similarly to the quadratic and quartic terms of the Taylor series expansion of the QEF in (\ref{expquad1}), the higher-order cumulants $\bK_k(\varphi(t))$ also grow asymptotically linearly with time $t$. For the formulation of the following theorems, we will need an auxiliary transformation
\begin{equation}
\label{funtrans}
    N^{[k]}(\tau)
    :=
    \left\{
    \begin{matrix}
        N(\tau)& {\rm if}\ k= 0\\
        N(-\tau)^{\rT}& {\rm if}\ k= 1
    \end{matrix}
    \right.
\end{equation}
of a matrix-valued function $N$ (of a  real variable), which is the identity transformation in the case $k=0$ (that is, $N^{[0]} = N$). Also, we denote by
\begin{equation}
\label{prod}
    \rprod_{k=1}^r N_k := N_1\x \ldots \x N_r
\end{equation}
the rightward-ordered  product of operators  or appropriately dimensioned matrices (with the order being essential because of noncommutativity). Furthermore, let
\begin{equation}
\label{fI}
    \fI(a,b)
    :=
    \left\{
    \begin{matrix}
        0 & {\rm if}\ a \< b\\
        1 & {\rm if}\ a > b
    \end{matrix}
    \right.
\end{equation}
indicate an inversion in the pair of integers $(a,b)$. The inversion indicator $\fI$ extends to consecutive pairs in an $r$-tuple of integers as
\begin{equation}
\label{fI}
    \fI(c_1, c_2, \ldots, c_r)
    :=
    (
        \fI(c_1,c_2),
        \fI(c_2,c_3),
        \ldots,
        \fI(c_{r-1},c_r)
    ).
\end{equation}
Also, for any binary $(r-2)$-index $\gamma \in \{0,1\}^{r-2}$, we denote by $\Delta_{r,\gamma}$ the number of permutations $c:=(c_1, \ldots, c_{r-1})$ of the set $\{1, \ldots, r-1\}$ with the inversion indicator $\gamma$:
\begin{equation}
\label{Delgam}
    \Delta_{r,\gamma}
    :=
    \#\{c \ {\rm is\ a\ permutation\ of}\ \{1,\ldots, r-1\}\ {\rm such\ that}\ \fI(c) = \gamma\}.
\end{equation}
Since such permutations, considered for all possible $\gamma \in \{0,1\}^{r-1}$, form a partitioning of the set of $(r-1)!$ permutations of $\{1, \ldots, r-1\}$, then
\begin{equation}
\label{sum}
    \sum_{\gamma \in \{0,1\}^{r-2}}
    \Delta_{r,\gamma} = (r-1)!.
\end{equation}

\begin{theorem}
\label{th:cumul}
Suppose the matrix $A$ in (\ref{A_B}) is Hurwitz, and the OQHO (\ref{dX}), driven by vacuum fields, is initialised at the invariant Gaussian state. Then for any $r\>2$ and any time $t\>0$, the $r$th cumulant (\ref{bK}) of the process $\varphi$ in (\ref{phi_psi}), with $\Pi\succcurlyeq 0$,  can be represented as
\begin{equation}
\label{bKphi}
    \bK_r(\varphi(t))
    =
    2^{r-1}
    \sum_{\gamma\in \{0,1\}^{r-2}}
    \Delta_{r,\gamma}
    \int_{[0,t]^r}
    \Tr
    \Big(
    \Pi S(t_1-t_2)
    \rprod_{j=2}^{r-1}
    \big(
        \Pi S^{[\gamma_j]}(t_j-t_{j+1})
    \big)
    \Pi S(t_1-t_r)^{\rT}
    \Big)
    \rd t_1\x \ldots \x \rd t_r,
\end{equation}
where $S$ is the steady-state quantum covariance function of the system variables in (\ref{S}), and use is made of the notation (\ref{funtrans}) and (\ref{prod}).
The sum in (\ref{bKphi}) is over binary $(r-2)$-indices $\gamma:= (\gamma_2, \ldots, \gamma_{r-1}) \in \{0,1\}^{r-2}$, and the corresponding coefficients $\Delta_{r,\gamma}$ are described by (\ref{fI})--(\ref{sum}).
\hfill$\square$
\end{theorem}
\begin{proof}
In what follows, for any given $r=1,2,3,\ldots$, a permutation
\begin{equation}
\label{kappa}
    \kappa
    :=
    (\kappa_1, \ldots, \kappa_{2r})
    =
    (j_1, k_1, \ldots, j_r, k_r)
\end{equation}
of the integers $1, \ldots, 2r$ is called \emph{regular} if it satisfies
\begin{equation}
    \label{jk}
        j_1 < j_2 <\ldots < j_{r-1} < j_r,
        \qquad
        j_1 < k_1,\
        j_2<k_2,\
         \ldots,\
         j_r< k_r.
\end{equation}
The set of regular permutations consists of $\frac{(2r)!}{r!2^r}=(2r-1)!!$ elements and is denoted by $\cP_r$. There is a one-to-one correspondence between $\cP_r$ and the set of all partitions of $\{1,\ldots,2r\}$ into two-element subsets. Note that $1$ is a fixed point for every regular permutation since $j_1 = 1$.  For example, the set of  regular permutations of $\{1,2,3,4\}$ is given by
\begin{equation*}
\label{cP2}
    \cP_2
    =
    \{(1,3,2,4),\ (1,4,2,3),\ (1,2,3,4)\}.
\end{equation*}
The Wick-Isserlis theorem\footnote{it is also used in Appendix~\ref{sec:covquad}}
  \cite[Theorem 1.28 on pp.~11--12]{J_1997} (see also \cite{I_1918} and \cite[p.~122]{M_2005})
  allows the product moment of self-adjoint quantum variables $\xi_1, \ldots, \xi_{2r}$, which satisfy CCRs and are  in a zero-mean Gaussian state, to be expressed in terms of their covariances as
\begin{equation}
\label{WIT}
    \bE
    \rprod_{k=1}^{2r}
    \xi_k
    =
    \sum_{\kappa \in \cP_r}
    \prod_{s=1}^r
    \bE(\xi_{j_s}\xi_{k_s}).
\end{equation}
By using (\ref{phi_psi}) and (\ref{psiZ}) 
and applying (\ref{WIT}) to the multi-point zero-mean Gaussian state of the entries of the process $Z$ at different moments of time, it   follows that
\begin{align}
\nonumber
    \bE
    (\varphi(t)^r)
    & =
    \bE
    \Big(
    \Big(
    \int_0^t
    \sum_{\ell=1}^n
    Z_\ell(\tau)^2
    \rd \tau
    \Big)^r
    \Big)\\
\nonumber
    & =
    \sum_{\ell_1, \ldots, \ell_r=1}^n
    \int_{[0,t]^r}
    \bE
    \rprod_{k=1}^r
    (Z_{\ell_k}(t_k)^2)
    \rd t_1\x \ldots \x \rd t_r
    \\
\nonumber
    & =
    \sum_{\ell_1, \ldots, \ell_r=1}^n
    \int_{[0,t]^r}
    \sum_{\kappa \in \cP_r}
    \prod_{s=1}^r
    K_{\ell_{d(j_s)} \ell_{d(k_s)}}(t_{d(j_s)}-t_{d(k_s)})
    \rd t_1\x \ldots \x \rd t_r
    \\
\label{EphirWIT}
    & =
    \sum_{\kappa \in \cP_r}
    \int_{[0,t]^r}
    \sum_{\ell_1, \ldots, \ell_r=1}^n
    \prod_{s=1}^r
    K_{\ell_{d(j_s)} \ell_{d(k_s)}}(t_{d(j_s)}-t_{d(k_s)})
    \rd t_1\x \ldots \x \rd t_r.
\end{align}
Here, $K$ denotes the steady-state quantum covariance function  of $Z$ in (\ref{psiZ}) which is related to that of the process  $X$ in (\ref{S}) by
\begin{align}
\nonumber
    \bE(Z(\sigma) Z(\tau)^{\rT})
    & =
    \sqrt{\Pi}S(\sigma-\tau)\sqrt{\Pi}\\
\label{K}
     & =:
     K(\sigma-\tau) = (K_{ab}(\sigma-\tau))_{1\< a, b\< n}.
\end{align}
Also, the function
\begin{equation}
\label{d}
    d(k)
    :=
    \lceil
        k/2
    \rceil
\end{equation}
in (\ref{EphirWIT}) maps the set $\{1,\ldots, 2r\}$ onto $\{1,\ldots, r\}$, where $\lceil \cdot\rceil$ denotes the ceiling function.  For any $j=1,\ldots, r$, the inverse image of the singleton $\{j\}$ under the map $d$  is the two-element set
$d^{-1}(j) = \{2j-1, 2j\}$. For any $\kappa \in \cP_r$ from (\ref{kappa}), the $2r$-tuple
\begin{align}
\nonumber
    \nu
    & :=
    (\nu_1, \ldots, \nu_{2r})
    :=
    d(\kappa)\\
\label{nu}
    & := (d(j_1), d(k_1), \ldots, d(j_r),d(k_r))
\end{align}
describes a permutation of the multi-set $\{1,1,2,2,\ldots, r-1, r-1, r,r\}$. Since the function $d$ in (\ref{d}) is nondecreasing, then $\nu$ inherits from $\kappa$ the property (\ref{jk}) in a nonstrict form:
\begin{equation}
\label{jknu}
        \nu_1 \< \nu_3 \<\ldots \< \nu_{2r-3} \< \nu_{2r-1},
        \qquad
        \nu_1 \< \nu_2,\
        \nu_3 \< \nu_4,\
         \ldots,\
         \nu_{2r-1}\< \nu_{2r}.
\end{equation}
Since the map $d$ is not injective,  some elements $\kappa \in \cP_r$ are mapped to the same permutation $\nu$ in (\ref{nu}), so that the inverse image $d^{-1}(\nu):= \{\kappa\in \cP_r:\ d(\kappa) = \nu\}$ of $\nu\in d(\cP_r)$ consists, in general,  of more than one regular permutation. In fact, only the identity permutation has a unique preimage: $d^{-1}(1,1,2,2,\ldots, r,r) = (1,2,\ldots, 2r-1,2r)$.   For example, the case $r=3$ is illustrated by Tabs.~\ref{tab:1} and \ref{tab:2} which also provide the numbers $\#d^{-1}(\nu)$ of elements in the inverse images.
\begin{table}[htbp]
\centering
\caption{The set $\cP_3$ of regular permutations $\kappa$ of $\{1,\ldots, 6\}$ and their images $\nu:= d(\kappa)$ under the map $d$ in (\ref{d})}
\label{tab:1}       
\begin{tabular}{l||ll|ll|ll||ll|ll|ll}
\hline\noalign{\smallskip}
 No. &
 $\kappa_1$ & $\kappa_2$ & 
$\kappa_3$ & $\kappa_4$ &  
$\kappa_5$ & $\kappa_6$ & 
$\nu_1$ & $\nu_2$ & 
$\nu_3$ & $\nu_4$ & 
$\nu_5$ & $\nu_6$ \\ 
\noalign{\smallskip}\hline\noalign{\smallskip}
  1  & 1  &   2  &   3  &   4  &   5  &   6 &      1  &   1  &   2  &   2  &   3 &    3                                \\
  2  & 1  &   2  &   3  &   5  &   4  &   6 &      1  &   1  &   2  &   3  &   2 &    3                               \\
  3  & 1  &   2  &   3  &   6  &   4  &   5 &      1  &   1  &   2  &   3  &   2 &    3                               \\
  4  & 1  &   3  &   2  &   4  &   5  &   6 &      1  &   2  &   1  &   2  &   3 &    3                               \\
  5  & 1  &   3  &   2  &   5  &   4  &   6 &      1  &   2  &   1  &   3  &   2 &    3                               \\
  6  & 1  &   3  &   2  &   6  &   4  &   5 &      1  &   2  &   1  &   3  &   2 &    3                               \\
  7  & 1  &   4  &   2  &   3  &   5  &   6 &      1  &   2  &   1  &   2  &   3 &    3                               \\
  8  & 1  &   4  &   2  &   5  &   3  &   6 &      1  &   2  &   1  &   3  &   2 &    3                               \\
  9  & 1  &   4  &   2  &   6  &   3  &   5 &      1  &   2  &   1  &   3  &   2 &    3                               \\
  10 & 1  &   5  &   2  &   3  &   4  &   6 &      1  &   3  &   1  &   2  &   2 &    3                               \\
  11 & 1  &   5  &   2  &   4  &   3  &   6 &      1  &   3  &   1  &   2  &   2 &    3                               \\
  12 & 1  &   5  &   2  &   6  &   3  &   4 &      1  &   3  &   1  &   3  &   2 &    2                               \\
  13 & 1  &   6  &   2  &   3  &   4  &   5 &      1  &   3  &   1  &   2  &   2 &    3                               \\
  14 & 1  &   6  &   2  &   4  &   3  &   5 &      1  &   3  &   1  &   2  &   2 &    3                               \\
  15 & 1  &   6  &   2  &   5  &   3  &   4 &      1  &   3  &   1  &   3  &   2 &    2                               \\
\noalign{\smallskip}\hline
\end{tabular}
\end{table}
\begin{table}[htbp]
\centering
\caption{The permutations $\nu\in d(\cP_3)$ of the multi-set $\{1,1,2,2,3,3\}$, their multiplicities $\# d^{-1}(\nu)$ and cycle structures}
\label{tab:2}       
\begin{tabular}{l||ll|ll|ll||l||l}
\hline\noalign{\smallskip}
 No.
&
$\nu_1$ & $\nu_2$ & 
$\nu_3$ & $\nu_4$ & 
$\nu_5$ & $\nu_6$ & 
 $\#d^{-1}(\nu)$ & cycles\\
\noalign{\smallskip}\hline\noalign{\smallskip}
  1  &      1  &   1  &   2  &   2  &   3 &    3 &      1 &
  (1), (2), (3)\\
  2  &      1  &   1  &   2  &   3  &   2 &    3 &      2 &
  (1), (2,3)\\
  3  &      1  &   2  &   1  &   2  &   3 &    3 &      2 &
  (1,2), (3)\\
  4  &      1  &   2  &   1  &   3  &   2 &    3 &      4 &
  (1,2,3)\\
  5  &      1  &   3  &   1  &   2  &   2 &    3 &      4 &
  (1,3,2)\\
  6  &      1  &   3  &   1  &   3  &   2 &    2 &      2 &
  (1,3), (2)\\
\noalign{\smallskip}\hline
\end{tabular}
\end{table}
For any $r\>1$, the relation (\ref{EphirWIT}) can be represented in terms of the permutations $\nu$ in (\ref{nu}) (with their multiplicities taken into account) as
\begin{equation}
\label{Ephirnu}
    \bE
    (\varphi(t)^r)
    =
    \sum_{\nu \in d(\cP_r)}
    \# d^{-1}(\nu)
    \int_{[0,t]^r}
    \sum_{\ell_1, \ldots, \ell_r=1}^n
    \prod_{s=1}^r
    K_{\ell_{\nu_{2s-1}} \ell_{\nu_{2s}}}(t_{\nu_{2s-1}}-t_{\nu_{2s}})
    \rd t_1\x \ldots \x \rd t_r.
\end{equation}
Now, any given permutation $\nu\in d(\cP_r)$ can be partitioned into cycles $c$ which pass through elements of the set $\{1, \ldots, r\}$ twice. Every  such cycle of period $p$ can be represented as a $p$-tuple $(c_1, c_2, \ldots, c_p)$ of pairwise different elements of the set $\{1, \ldots, r\}$ (see, for example, the rightmost column of Tab.~\ref{tab:2}). In terms of these cycles (which depend on a particular permutation $\nu \in d(\cP_r)$), the product in (\ref{Ephirnu}) takes the form
\begin{equation}
\label{Kprod}
    \prod_{s=1}^r
    K_{\ell_{\nu_{2s-1}} \ell_{\nu_{2s}}}(t_{\nu_{2s-1}}-t_{\nu_{2s}})
    =
    \prod_c
    \prod_{j=1}^p
    K_{\ell_{c_j} \ell_{c_{j+1}}}^{[\gamma_j]}(t_{c_j}-t_{c_{j+1}}).
\end{equation}
Here, for a given cycle $c:= (c_1, \ldots, c_p)$ of period $p$, the convention $c_{p+1}:= c_1$ is used, and $\gamma_1, \ldots, \gamma_p$ are auxiliary variables
\begin{equation}
\label{flip}
    \gamma_j
    :=
    \fI(c_j,c_{j+1})
    =
    \left\{
    \begin{matrix}
        0 & {\rm if}\ c_j < c_{j+1}\\
        1 & {\rm if}\ c_j > c_{j+1}
    \end{matrix}
    \right.
\end{equation}
which, in accordance with (\ref{fI}), indicate inversions for consecutive elements of the cycle $c$
(we let $\gamma_1:=0$ in the case $p=1$); see Fig.~\ref{fig:cycles}.
\begin{figure}
\centering
\begin{tikzpicture}
\tikzset{vertex/.style = {shape=circle,draw,minimum size=1mm}}
\tikzset{edge/.style = {->,> = latex'}}
\node at (-1,0) {1)};
\node at (-1,-1.5) {2)};
\node at (-1,-3) {3)};
\node at (-1,-4.5) {4)};
\node at (-1,-6) {5)};
\node at (-1,-7.5) {6)};
\node[vertex] (11) at  (0,0)  {\small$1$};
\node[vertex] (12) at  (1.5,0)  {\small$1$};
\node[vertex] (13) at  (3,0)  {\small$2$};
\node[vertex] (14) at  (4.5,0)  {\small$2$};
\node[vertex] (15) at  (6,0)  {\small$3$};
\node[vertex] (16) at  (7.5,0) {\small$3$};

\draw[edge] (11) to (12);
\draw[edge] (12) to[bend right] (11);
\draw[edge] (13) to (14);
\draw[edge] (14) to[bend right] (13);
\draw[edge] (15) to (16);
\draw[edge] (16) to[bend right] (15);

\node[vertex] (21) at  (0,-1.5)  {\small$1$};
\node[vertex] (22) at  (1.5,-1.5)  {\small$1$};
\node[vertex] (23) at  (3,-1.5)  {\small$2$};
\node[vertex] (24) at  (4.5,-1.5)  {\small$3$};
\node[vertex] (25) at  (6,-1.5)  {\small$2$};
\node[vertex] (26) at  (7.5,-1.5) {\small$3$};
\draw[edge] (21) to (22);
\draw[edge] (22) to[bend right] (21);
\draw[edge] (23) to (24);
\draw[edge] (24) to[bend left] (26);
\draw[edge] (26) to (25);
\draw[edge] (25) to[bend right] (23);

\node[vertex] (31) at  (0,-3)  {\small$1$};
\node[vertex] (32) at  (1.5,-3)  {\small$2$};
\node[vertex] (33) at  (3,-3)  {\small$1$};
\node[vertex] (34) at  (4.5,-3)  {\small$2$};
\node[vertex] (35) at  (6,-3)  {\small$3$};
\node[vertex] (36) at  (7.5,-3) {\small$3$};
\draw[edge] (31) to (32);
\draw[edge] (32) to[bend left] (34);
\draw[edge] (34) to (33);
\draw[edge] (33) to[bend right] (31);
\draw[edge] (35) to (36);
\draw[edge] (36) to[bend right] (35);

\node[vertex] (41) at  (0,-4.5)  {\small$1$};
\node[vertex] (42) at  (1.5,-4.5)  {\small$2$};
\node[vertex] (43) at  (3,-4.5)  {\small$1$};
\node[vertex] (44) at  (4.5,-4.5)  {\small$3$};
\node[vertex] (45) at  (6,-4.5)  {\small$2$};
\node[vertex] (46) at  (7.5,-4.5) {\small$3$};
\draw[edge] (41) to (42);
\draw[edge] (42) to[bend left] (45);
\draw[edge] (45) to (46);
\draw[edge] (46) to[bend right] (44);
\draw[edge] (44) to(43);
\draw[edge] (43) to[bend right] (41);

\node[vertex] (51) at  (0,-6)  {\small$1$};
\node[vertex] (52) at  (1.5,-6)  {\small$3$};
\node[vertex] (53) at  (3,-6)  {\small$1$};
\node[vertex] (54) at  (4.5,-6)  {\small$2$};
\node[vertex] (55) at  (6,-6)  {\small$2$};
\node[vertex] (56) at  (7.5,-6) {\small$3$};
\draw[edge] (51) to (52);
\draw[edge] (52) to[bend left] (56);
\draw[edge] (56) to (55);
\draw[edge] (55) to[bend right] (54);
\draw[edge] (54) to(53);
\draw[edge] (53) to[bend right] (51);

\node[vertex] (61) at  (0,-7.5)  {\small$1$};
\node[vertex] (62) at  (1.5,-7.5)  {\small$3$};
\node[vertex] (63) at  (3,-7.5)  {\small$1$};
\node[vertex] (64) at  (4.5,-7.5)  {\small$3$};
\node[vertex] (65) at  (6,-7.5)  {\small$2$};
\node[vertex] (66) at  (7.5,-7.5) {\small$2$};
\draw[edge] (61) to (62);
\draw[edge] (62) to[bend left] (64);
\draw[edge] (64) to (63);
\draw[edge] (63) to[bend right] (61);
\draw[edge] (65) to (66);
\draw[edge] (66) to[bend right] (65);
\end{tikzpicture}
\caption{These directed graphs represent the permutations $\nu$ of the multi-set $\{1,1,2,2,3,3\}$ in  Tab.~\ref{tab:2} (for the case $r=3$) and their cycle structure. The straight edges in each of the graphs depict the pairs $(c_j,c_{j+1})$ in (\ref{flip}) and are oriented rightwards ($\gamma_j=0$) or leftwards ($\gamma_j=1$), with the latter indicating the inversions in these pairs. Each of the curved edges represents a transition to the remaining occurrence  of the current element of the cycle in the permutation. The fourth and fifth permutations are monocyclic. According to the second last column of Tab.~\ref{tab:2}, their total multiplicity is $8$. }
\label{fig:cycles}
\end{figure}
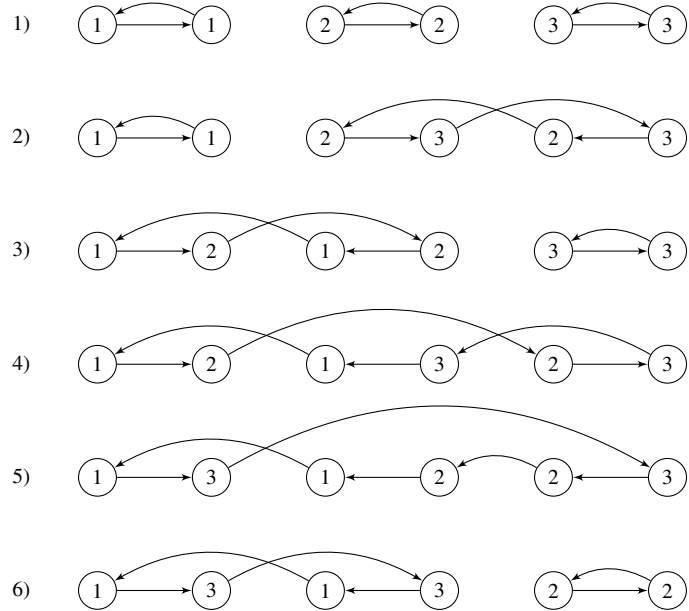
Also, use is made of the notation
\begin{equation}
\label{Kflip}
    K_{\ell_{c_j} \ell_{c_{j+1}}}^{[\gamma_j]}(t_{c_j}-t_{c_{j+1}})
    :=
    \left\{
    \begin{matrix}
        K_{\ell_{c_j} \ell_{c_{j+1}}}(t_{c_j}-t_{c_{j+1}}) & {\rm if}\ c_j < c_{j+1}\\
        K_{\ell_{c_{j+1}} \ell_{c_j}}(t_{c_{j+1}}-t_{c_j}) & {\rm if}\ c_{j+1}< c_j
    \end{matrix}
    \right..
\end{equation}
Since the cycles of a given permutation $\nu$ partition the set $\{1, \ldots, r\}$, then, in view of (\ref{flip}) and (\ref{Kflip}),   the summation of (\ref{Kprod}) over the independent indices $\ell_1, \ldots, \ell_r$ leads to
\begin{align}
\nonumber
    \sum_{\ell_1, \ldots, \ell_r=1}^n
    \prod_{s=1}^r
    K_{\ell_{\nu_{2s-1}} \ell_{\nu_{2s}}}(t_{\nu_{2s-1}}-t_{\nu_{2s}})
    & =
    \sum_{\ell_1, \ldots, \ell_r=1}^n
    \prod_c
    \prod_{j=1}^p
    K_{\ell_{c_j} \ell_{c_{j+1}}}^{[\gamma_j]}(t_{c_j}-t_{c_{j+1}})\\
\label{sumKprod}
    & =
    \prod_c
    \Tr
    \rprod_{j=1}^p
    K^{[\gamma_j]}(t_{c_j}-t_{c_{j+1}}),
\end{align}
where the last equality also employs (\ref{funtrans}) and (\ref{prod}).  By a similar reasoning, the integration of the right-hand side of  (\ref{sumKprod}) over $t_1, \ldots, t_r$ yields
\begin{equation}
\label{intsumKprod}
    \int_{[0,t]^r}
    \prod_c
    \Tr
    \rprod_{j=1}^p
    K^{[\gamma_j]}(t_{c_j}-t_{c_{j+1}})
    \rd t_1 \x \ldots \x \rd t_r
    =
    \prod_c
    \int_{[0,t]^p}
        \Tr
    \rprod_{j=1}^p
    K^{[\gamma_j]}(t_j-t_{j+1})
    \rd t_1\x \ldots \x \rd t_p,
\end{equation}
where $t_{p+1} := t_1$. Substitution of (\ref{intsumKprod}) into (\ref{Ephirnu}) leads to
\begin{equation}
\label{Ephirnu1}
    \bE
    (\varphi(t)^r)
    =
    \sum_{\nu \in d(\cP_r)}
    \# d^{-1}(\nu)
    \prod_c
    \int_{[0,t]^p}
        \Tr
    \rprod_{j=1}^p
    K^{[\gamma_j]}(t_j-t_{j+1})
    \rd t_1\x \ldots \x \rd t_p.
\end{equation}
A combination of (\ref{Ephirnu1}) with (\ref{bKP}) and tedious combinatorial considerations (omitted here for brevity)  show that the contributions to the $r$th cumulant $\bK_r(\varphi(t))$  from all those permutations $\nu$, which  consist of at least two cycles,  cancel out in the sense that
\begin{align*}
    \sum_{\nu \in d(\cP_r)\ {\rm with\ at\ least\ two\ cycles}} &
    \# d^{-1}(\nu)
    \prod_c
    \int_{[0,t]^p}
        \Tr
    \rprod_{j=1}^p
    K^{[\gamma_j]}(t_j-t_{j+1})
    \rd t_1\x \ldots \x \rd t_p\\
    & =
    r!
    \sum_{k=2}^r
    \frac{(-1)^k}{k}
    \sum_{j_1,\ldots, j_k\> 1:\ j_1+\ldots+j_k = r}\
    \prod_{s=1}^k
    \frac{\bE(\varphi(t)^{j_s})}{j_s!}.
\end{align*}
Hence, only \emph{monocyclic} permutations $\nu \in d(\cP_r)$,  consisting of a single cycle of period $r$ (the set of such permutations is denoted by $\cM_r$), contribute to the cumulant:
\begin{equation}
\label{bKr}
    \bK_r(\varphi(t))
    =
    \sum_{\nu \in \cM_r}
    \# d^{-1}(\nu)
    \int_{[0,t]^r}
        \Tr
    \rprod_{j=1}^r
    K^{[\gamma_j]}(t_j-t_{j+1})
    \rd t_1\x \ldots \x \rd t_r.
\end{equation}
The cycle, associated with a monocyclic permutation $\nu \in \cM_r$, has the form $(1, c_2, \ldots, c_r)$, where $(c_2, \ldots, c_r)$ is an arbitrary permutation of the set $\{2, \ldots, r\}$.
For any such $(c_2, \ldots, c_r)$, all monocyclic permutations $\nu\in \cM_r$ with the cycle $(1, c_2, \ldots, c_r)$ are described by
\begin{equation}
\label{nupp}
    \nu = (1,c_2,1,c_r, p_1, \ldots, p_{r-2}),
\end{equation}
where $p_1, \ldots, p_{r-2}$ are pairs of integers which are obtained by ordering the pairs
\begin{equation}
\label{abc}
    (a_j,b_j)
    :=
    (
        \min(c_j,c_{j+1}),
        \max(c_j,c_{j+1})
    )
\end{equation}
for all $j=2,\ldots, r-1$, so that their first entries $a_2, \ldots, a_{r-1}$ are ordered in ascending order  according to (\ref{jknu}). The ordering is delivered by all those permutations $f_1, \ldots, f_{r-2}$ of the set $\{2, \ldots, r-1\}$ which satisfy
\begin{equation}
\label{aforder}
    a_{f_1}\< \ldots \< a_{f_{r-2}}
\end{equation}
and the trailing pairs in (\ref{nupp}) are given in terms of (\ref{abc}) by
\begin{equation}
\label{pab}
    p_j := (a_{f_j}, b_{f_j}),
    \qquad
    j = 1, \ldots, r-2.
\end{equation}
Such a permutation  $f_1, \ldots, f_{r-2}$ is unique only if $a_2, \ldots, a_{r-1}$ are pairwise different. If there are $\mu$ different pairs of identical elements among $a_2, \ldots, a_{r-1}$, then the number of permutations $f_1, \ldots, f_{r-2}$ which secure (\ref{aforder}) is $2^{\mu}$. However, in both scenarios, there are $2^{r-1}$ permutations $\kappa \in \cP_r$ such that the corresponding $\nu$ in (\ref{nu}) is monocyclic with the given cycle $(1, c_2,\ldots, c_r)$. Now, since the matrix $    \rprod_{j=1}^r
    K^{[\gamma_j]}(t_j-t_{j+1})
$ on the right-hand side of (\ref{bKr}) depends on a monocyclic permutation $\nu$ in (\ref{nupp})--(\ref{pab})  only through the inversion indicators $\gamma_1, \ldots, \gamma_r$ in (\ref{flip}), which are completely specified by $c_2, \ldots, c_r$, then
\begin{equation}
\label{bKr1}
    \bK_r(\varphi(t))
    =
    2^{r-1}
    \sum_{(c_2, \ldots, c_r)}
    \int_{[0,t]^r}
    \Tr
    \rprod_{j=1}^r
    K^{[\gamma_j]}(t_j-t_{j+1})
    \rd t_1\x \ldots \x \rd t_r,
\end{equation}
where the sum is over all $(r-1)!$ permutations $(c_2, \ldots, c_r)$ of the set $\{2, \ldots, r\}$. The summation in (\ref{bKr1}) can be reduced to that over the binary variables $\gamma_2, \ldots, \gamma_{r-1}$ as
\begin{equation}
\label{bKr2}
    \bK_r(\varphi(t))
    =
    2^{r-1}
    \sum_{\gamma_2, \ldots, \gamma_{r-1} = 0,1}
    \Delta_{r,\gamma_2, \ldots, \gamma_{r-1}}
    \int_{[0,t]^r}
    \Tr
    \Big(
    K(t_1-t_2)
    \rprod_{j=2}^{r-1}
    K^{[\gamma_j]}(t_j-t_{j+1})
    K(t_1-t_r)^{\rT}
    \Big)
    \rd t_1\x \ldots \x \rd t_r,
\end{equation}
since $\gamma_1 = 0$ and $\gamma_r = 1$ in view of (\ref{flip}) and $c_1=1< c_2$ and $c_r>c_{r+1}=1$.  Here, in accordance with (\ref{Delgam}), the coefficient  $\Delta_{r,\gamma_2, \ldots, \gamma_{r-1}}$ is the number of those permutations $(c_2, \ldots, c_r)$ of the set $\{2, \ldots, r\}$  which have the inversion indicators $\gamma_2, \ldots, \gamma_{r-1}$. Therefore, the sum of these coefficients is given by (\ref{sum}). From (\ref{K}) and the symmetry of the matrix $\Pi$, it follows that
\begin{align}
\nonumber
    \Tr
    \rprod_{j=1}^r
    K^{[\gamma_j]}(t_j-t_{j+1})
    & =
    \Tr
    \rprod_{j=1}^r
    \big(
        \sqrt{\Pi}
        S^{[\gamma_j]}(t_j-t_{j+1})
        \sqrt{\Pi}
    \big)\\
\label{KS}
    & =
    \Tr
    \Big(
    \Pi S(t_1-t_2)
    \rprod_{j=2}^{r-1}
    \big(
        \Pi S^{[\gamma_j]}(t_j-t_{j+1})
    \big)
    \Pi S(t_1-t_r)^{\rT}
    \Big).
\end{align}
Substitution of (\ref{KS}) into (\ref{bKr2}) leads to (\ref{bKphi}), thus completing the proof.
\end{proof}

Although Theorem~\ref{th:cumul} is formulated and proved for the case $\Pi\succcurlyeq 0$, the relation (\ref{bKphi}) remains valid irrespective of the matrix $\Pi$ being positive semi-definite.

In contrast to covariance functions of classical random processes, the quantum covariance function $S$ in (\ref{S}) is not invariant under the transformation (\ref{funtrans}) since $S^{[1]}(\tau)=\overline{S(\tau)} \ne S(\tau)$ in general. Therefore, the integral in (\ref{bKphi}) has a more complicated structure than the $r$-fold convolution of the function $\Pi S$ with itself.
For example, in the case $r=3$, the sum in (\ref{bKphi}) is over a binary index and consists of two terms. In view of (\ref{sum}) and the   invariance of the matrix trace with respect to the transpose and cyclic permutations, this leads to
\begin{align}
\nonumber
    \bK_3(\varphi(t))
     = &
    4\Big(\Delta_{3,0}
    \int_{[0,t]^3}
    \Tr
    \big(
    \Pi
        S(t_1-t_2)
        \Pi
        S(t_2-t_3)
        \Pi
        S(t_1-t_3)^{\rT}
    \big)
    \rd t_1
    \rd t_2
    \rd t_3\\
\nonumber
    & +
    \Delta_{3,1}
    \int_{[0,t]^3}
    \Tr
    \big(
    \Pi
        S(t_1-t_2)
        \Pi
        S(t_3-t_2)^{\rT}
        \Pi
        S(t_1-t_3)^{\rT}
    \big)
    \rd t_1
    \rd t_2
    \rd t_3\Big)\\
\nonumber
    = &
    4 (\Delta_{3,0}+\Delta_{3,1})
    \int_{[0,t]^3}
    \Tr
    \big(
    \Pi
        S(t_1-t_2)
        \Pi
        S(t_2-t_3)
        \Pi
        S(t_1-t_3)^{\rT}
    \big)
    \rd t_1
    \rd t_2
    \rd t_3\\
\label{bK3}
    = &
    8
    \int_{[0,t]^3}
    \Bra
        \Pi,
        S(t_1-t_2)
        \Pi
        S(t_2-t_3)
        \Pi
        S(t_1-t_3)^{\rT}
    \Ket
    \rd t_1
    \rd t_2
    \rd t_3.
\end{align}
For larger values of $r$, the dependence of the coefficient $\Delta_{r,\gamma}$ in (\ref{bKphi}) on the multi-index $\gamma \in \{0,1\}^{r-2}$  is fairly complicated and manifests fractal-like fluctuations, as illustrated by Fig.~\ref{fig:Delgam}.
\begin{figure}[thpb]
      \centering
      \includegraphics[width=150mm]{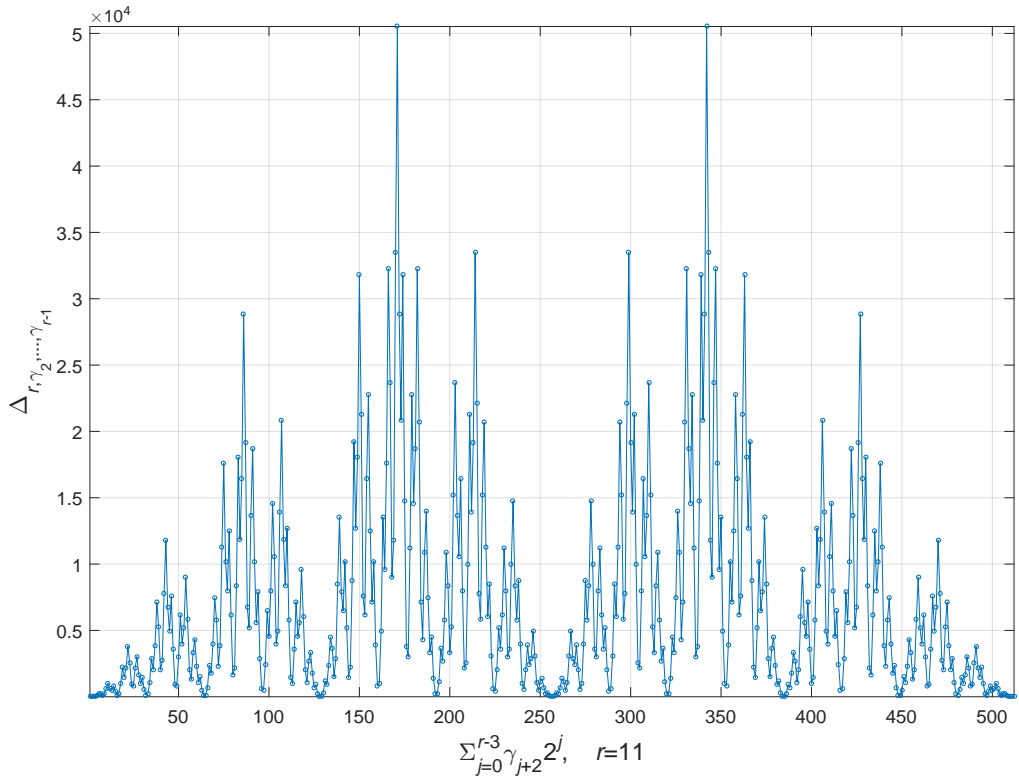}
      \caption{The dependence of the coefficient $\Delta_{r,\gamma}$ in (\ref{bKphi}) on the binary $(r-2)$-index $\gamma$ in the case $r=11$. The multi-index encodes the integers $0,\ldots, 2^{r-2}-1$ over the horizontal axis of the graph. }
      \label{fig:Delgam}
   \end{figure}

The following corollary of Theorem~\ref{th:cumul} describes the infinite-horizon asymptotic behaviour of the cumulants of arbitrary order, including its frequency-domain representation. To this end, we denote by $D$ the spectral density of the quantum covariance function (\ref{S}):
\begin{equation}
\label{D}
    D(\lambda)
    :=
    \int_{-\infty}^{+\infty}
    S(\tau) \re^{-i\lambda\tau}
    \rd \tau.
\end{equation}
Due to the Bochner-Khinchin criterion of positive semi-definiteness of Hermitian kernels \cite{GS_2004} and the standard properties of the Fourier transforms, $D(\lambda) = D(\lambda)^*\succcurlyeq 0$ for all frequencies $\lambda \in \mR$. Furthermore, $D$ admits the factorization
\begin{equation}
\label{DG}
    D(\lambda) = G(i\lambda) \Omega G(i\lambda)^*
\end{equation}
in terms of the quantum Ito matrix $\Omega$ from (\ref{WW}) and the transfer function $G$ from $W$ to $X$ in (\ref{dX}) associated with the matrix pair $(A,B)$ by
\begin{equation}
\label{G}
    G(s):=
    \int_0^{+\infty}
    \re^{-st}
    \re^{-tA}
    B
    \rd t
    =
    (sI_n - A)^{-1}B
\end{equation}
in the closed right half-plane $\Re s\>0$ (since $A$ is Hurwitz).
The transformation (\ref{funtrans}), applied to the covariance function $S$,  carries over to the spectral density $D$ in (\ref{D}) as
\begin{align*}
\nonumber
    S^{[1]}(\tau)
    & =
    S(-\tau)^{\rT}\\
\nonumber
    & =
    \frac{1}{2\pi}
    \int_{-\infty}^{+\infty}
    \re^{-i\lambda\tau}
    D(\lambda)^{\rT}
    \rd \lambda\\
    & =
    \frac{1}{2\pi}
    \int_{-\infty}^{+\infty}
    \re^{i\lambda\tau}
    D(-\lambda)^{\rT}
    \rd \lambda
    =
    \frac{1}{2\pi}
    \int_{-\infty}^{+\infty}
    \re^{i\lambda\tau}
    D^{[1]}(\lambda)
    \rd \lambda,
\end{align*}
with the latter being the inverse Fourier transform of $D^{[1]}$. Also, similarly to (\ref{prod}), for appropriately dimensioned matrix-valued functions $f_1, \ldots, f_r$ on the real line, their convolution is denoted by
\begin{equation}
\label{conv}
    \conv_{k=1}^r
    f_k
    :=
    f_1 * \ldots * f_r.
\end{equation}

\begin{theorem}
\label{th:asycumul}
Under the conditions of Theorem~\ref{th:cumul}, 
for any $r\>2$, the asymptotic growth rate of the $r$th cumulant in (\ref{bKphi})  can be computed as
\begin{align}
\nonumber
    \lim_{t\to +\infty}
    \Big(
        \frac{1}{t}
        \bK_r(\varphi(t))
    \Big)
    & =
    2^{r-1}
    \sum_{\gamma\in \{0,1\}^{r-2}}
    \Delta_{r,\gamma}
    \Tr
    \Big(
    \Pi S
    *
    \conv_{j=2}^{r-1}
    \big(
        \Pi S^{[\gamma_j]}
    \big)
    *
    \Pi S^{[1]}
    \Big)(0)\\
\label{bKphiasy}
    & =
    \frac{2^{r-2}    }{\pi}
    \sum_{\gamma\in \{0,1\}^{r-2}}
    \Delta_{r,\gamma}
    \int_{-\infty}^{+\infty}
        \Tr
    \Big(
    \Pi D(\lambda)
    \rprod_{j=2}^{r-1}
    \big(
        \Pi D^{[\gamma_j]}(\lambda)
    \big)
    \Pi D^{[1]}(\lambda)
    \Big)
    \rd \lambda,
\end{align}
where the notation (\ref{conv}) is used, 
and $D$ is the spectral density of the steady-state quantum covariance function $S$ in (\ref{S}) given by  (\ref{D})--(\ref{G}). \hfill$\square$
\end{theorem}
\begin{proof}
The relations (\ref{bKphiasy}) are obtained by applying Lemma~\ref{lem:conv} of Appendix~\ref{sec:conv}   to the integrals in (\ref{bKphi}) with the functions $f_1 := \Pi S$, $f_r:= \Pi S^{[1]}$ and $f_k:= \Pi S^{[\gamma_k]}$ for $k=2,\ldots, r-1$ whose Fourier transforms are $\Pi D$, $\Pi D^{[1]}$ and $\Pi D^{[\gamma_k]}$, respectively.
\end{proof}

Note that in the case $r=2$, the first of the equalities (\ref{bKphiasy}) reproduces (\ref{PiT}).
In application to the third cumulant in (\ref{bK3}), Theorem~\ref{th:asycumul} yields
\begin{align*}
\nonumber
    \lim_{t\to +\infty}
    \Big(
        \frac{1}{t}
        \bK_3(\varphi(t))
    \Big)
     & =
    8
    \Tr
    \big(
        \Pi S
        *
        \Pi S *
        \Pi S^{[1]}
    \big)(0)\\
\nonumber
    & = 8
    \int_{\mR^2}
    \Bra
        \Pi,
        S(\sigma)
        \Pi
        S(\tau-\sigma)
        \Pi
        S(\tau)^{\rT}
    \Ket
    \rd \sigma
    \rd \tau\\
    & =
    \frac{4}{\pi}
    \int_{-\infty}^{+\infty}
    \Bra
        \Pi,
        (D(\lambda) \Pi)^2 D(-\lambda)^{\rT}
    \Ket
    \rd \lambda.
\end{align*}
In comparison with the classical case, the representation (\ref{bKphiasy}) involves the sum of terms whose number grows exponentially with $r$. The complicated structure of these terms, which are controlled by successive inversions in permutations, and the combinatorial nature of the coefficients $\Delta_{r,\gamma}$ give rise to the problem of developing a recurrence relation for the cumulant rates.

\section{Large deviations estimates}\label{sec:large}

Theorems~\ref{th:cumul} and \ref{th:asycumul} of the previous section can be employed for obtaining large deviations estimates for the process $\varphi$ in (\ref{phi_psi}).\footnote{Such estimates are physically meaningful in the case $\Pi\succcurlyeq  0$ when the self-adjoint quantum variable $\varphi(t)$ is a positive semi-definite operator.} More precisely, for any given time $t>0$, application of the Cramer inequality \cite{S_1996} (see also \cite[Section~3.5]{DE_1997}) to the probability distribution $E_t$ of $\varphi(t)$ yields the inequality
\begin{equation}
\label{cramer}
    \frac{1}{t}
    \ln
    E_t([\eps t, +\infty))
    \<
    \inf_{\theta\>0}
    \Big(
        \frac{1}{t}
        \ln \Xi_{\theta}(t) - \theta \eps
    \Big)
\end{equation}
in terms of the QEF (\ref{QEF}) for any $\eps>0$. Similarly to the classical case, the function being minimised in (\ref{cramer}) is convex with respect to $\theta$ and vanishes at $\theta=0$. The following theorem leads to an upper bound for the right-hand side of (\ref{cramer})\footnote{which is the negative of the Legendre transform of $\frac{1}{t}
        \ln \Xi_{\theta}(t)$ as a function of $\theta\>0$}, uniform in time $t$. For its formulation, we need an auxiliary function $N$ associated with the quantum covariance function $S$ in (\ref{S}) by
\begin{equation}
\label{N}
    N(\tau)
    :=
    \|\sqrt{\Pi} S(\tau)\sqrt{\Pi}\|,
\end{equation}
where $\|\cdot\|$ is the $\ell_2$-induced operator norm of a matrix. Note that $N$ is an $\mR_+$-valued integrable function which is even in view of the second equality in (\ref{SVLambda}) and the invariance of the matrix norm $\|\cdot\|$ under the complex conjugate transpose:
\begin{equation}
\label{Nsymm}
    N(-\tau)
    =
    \|(\sqrt{\Pi} S(\tau)\sqrt{\Pi})^*\|
    =
    N(\tau).
\end{equation}

\begin{theorem}
\label{th:dev}
Suppose the conditions of Theorem~\ref{th:cumul} are satisfied. Also, let the risk-sensitivity parameter $\theta$ belong to the interval
\begin{equation}
\label{thetarange}
    0 \< \theta < \frac{1}{2\|F\|_{\infty}},
\end{equation}
where $\|F\|_{\infty}$ is the $L_{\infty}$-norm of the Fourier transform $F$ of the function $N$ from (\ref{N}) given by $F(\lambda):= \int_{-\infty}^{+\infty} N(\tau) \re^{-i\lambda \tau} \rd \lambda$. Then the QEF (\ref{QEF}) admits the upper bound
\begin{equation}
\label{upQEF}
    \sup_{t>0}
    \Big(
        \frac{1}{t}
        \ln \Xi_{\theta}(t)
    \Big)
    \<
    -\frac{n}{4\pi}
    \int_{-\infty}^{+\infty}
    \ln
    (1 - 2\theta F(\lambda))
    \rd
    \lambda.
\end{equation}
\hfill$\square$
\end{theorem}
\begin{proof}
Similarly to (\ref{Nsymm}), from the invariance of the operator norm $\|\cdot\|$ with respect to the transpose and the symmetry of the function $N$,  it follows that
\begin{equation}
\label{Nsymm2}
    \|\sqrt{\Pi} S^{[1]}(\tau)\sqrt{\Pi}\|
    =
    \|(\sqrt{\Pi} S(-\tau)\sqrt{\Pi})^{\rT}\|
    =
    N(\tau),
\end{equation}
where use is made of the transformation (\ref{funtrans}). A combination of (\ref{Nsymm2}) with the submultiplicativity  of the operator norm leads to
\begin{align}
\nonumber
    \Big|
        \int_{[0,t]^r}
        &
    \Tr
    \Big(
    \Pi S(t_1-t_2)
    \rprod_{j=2}^{r-1}
    \big(
        \Pi S^{[\gamma_j]}(t_j-t_{j+1})
    \big)
    \Pi S(t_1-t_r)^{\rT}
    \Big)
    \rd t_1\x \ldots \x \rd t_r
    \Big|\\
\nonumber
    & \<
    n
        \int_{[0,t]^r}
    \Big\|
        \sqrt{\Pi} S(t_1-t_2)\sqrt{\Pi}
    \rprod_{j=2}^{r-1}
    \big(
        \sqrt{\Pi}  S^{[\gamma_j]}(t_j-t_{j+1})
        \sqrt{\Pi}
    \big)
    \sqrt{\Pi} S(t_1-t_r)^{\rT} \sqrt{\Pi}
    \Big\|
    \rd t_1\x \ldots \x \rd t_r
    \\
\nonumber
    & \<
    n
        \int_{[0,t]^r}
    \prod_{j=1}^r
    N(t_j-t_{j+1})
    \rd t_1\x \ldots \x \rd t_r\\
\label{upint}
    & \<
    nt N^{*r}(0)
    =
    \frac{nt}{2\pi}
    \int_{-\infty}^{+\infty}
    F(\lambda)^r
    \rd
    \lambda,
\end{align}
where $N^{*r}$ denotes the $r$-fold convolution of the function $N$ with itself.
The last inequality in (\ref{upint}) follows from the nonnegativeness of $N$ and the identity (\ref{fg1}) in the proof of Lemma~\ref{lem:conv}, while the first equality uses the  upper bound $|\Tr K| \< n \|K\|$ which holds for any matrix $K\in \mC^{n\x n}$.  In view of (\ref{upint}) and (\ref{sum}), the $r$th cumulant in (\ref{bKphi}) satisfies the inequality
\begin{align}
\nonumber
    \bK_r(\varphi(t))
    & \<
    2^{r-1}
    \sum_{\gamma\in \{0,1\}^{r-2}}
    \Delta_{r,\gamma}
    \frac{nt}{2\pi}
    \int_{-\infty}^{+\infty}
    F(\lambda)^r
    \rd
    \lambda\\
\label{upbK}
    & =
    (r-1)!
    2^{r-2}
    \frac{nt}{\pi}
    \int_{-\infty}^{+\infty}
    F(\lambda)^r
    \rd
    \lambda
\end{align}
for any $r=1,2,3,\ldots$. Assuming that $\theta$ is nonnegative and sufficiently small in the sense of (\ref{thetarange}),  substitution of (\ref{upbK}) into (\ref{expquad1}) leads to
\begin{align}
\nonumber
    \ln \Xi_{\theta}(t)
    & \<
    \sum_{r=1}^{+\infty}
    \frac{\theta^r}{r!}
    (r-1)!
    2^{r-2}
    \frac{nt}{\pi}
    \int_{-\infty}^{+\infty}
    F(\lambda)^r
    \rd
    \lambda\\
\nonumber
    & =
    \frac{nt}{4\pi}
    \sum_{r=1}^{+\infty}
    \frac{1}{r}
    \int_{-\infty}^{+\infty}
    (2\theta F(\lambda))^r
    \rd
    \lambda\\
\label{upQEF1}
    & =
    -\frac{nt}{4\pi}
    \int_{-\infty}^{+\infty}
    \ln
    (1 - 2\theta F(\lambda))
    \rd
    \lambda.
\end{align}
The inequality (\ref{upQEF}) can now be obtained by dividing both sides of (\ref{upQEF1}) by $t>0$ and using the fact that $t$ is otherwise arbitrary.
\end{proof}

A combination of (\ref{cramer}) with (\ref{upQEF}) of Theorem~\ref{th:dev} leads to
\begin{equation}
\label{cramer1}
    \sup_{t>0}
    \Big(
    \frac{1}{t}
    \ln
    E_t([\eps t, +\infty))
    \Big)
    \<
    \inf_{0\< \theta< \frac{1}{2\|F\|_{\infty}}}
    \Big(
    -\frac{n}{4\pi}
    \int_{-\infty}^{+\infty}
    \ln
    (1 - 2\theta F(\lambda))
    \rd
    \lambda
    -
        \theta \eps
    \Big).
\end{equation}
The function under minimization in (\ref{cramer1}) is strictly convex with respect to $\theta$ in the interval  (\ref{thetarange}) and hence, has a unique minimum. The optimal value of $\theta$ can be  found from the equation
\begin{equation}
\label{opt}
    -\frac{n}{4\pi}
    \d_{\theta}
    \int_{-\infty}^{+\infty}
    \ln
    (1 - 2\theta F(\lambda))
    \rd
    \lambda
    =
    \frac{n}{2\pi}
    \int_{-\infty}^{+\infty}
    \frac{F(\lambda)}
    {1 - 2\theta F(\lambda)}
    \rd
    \lambda
    =
    \eps
\end{equation}
whose left-hand side is a strictly increasing function of $\theta$.  The solution $\theta\>0$ exists and is unique if the ``scale'' parameter $\eps$ is large enough in the sense that
\begin{equation}
\label{epsbig}
    \frac{\eps}{n}
    \>
    \frac{1}{2\pi}
    \int_{-\infty}^{+\infty}
    F(\lambda)
    \rd \lambda
    =
    N(0)
    =
    \|\sqrt{\Pi} (P+i\Theta)\sqrt{\Pi}\|.
\end{equation}
Note that (\ref{upQEF}), (\ref{cramer1}) and (\ref{opt}) remain valid if the function $N$ in (\ref{N}) is replaced with its even upper estimate. Since the quantum covariance function $S$ decays exponentially fast at infinity, such an upper bound for $N$ can be found among functions with a rational Fourier transform $F$. In this case, the integral  in (\ref{cramer1}) lends itself to effective evaluation \cite{MG_1990} as does the solution of (\ref{opt}).

\begin{theorem}
\label{th:devup}
Suppose the conditions of Theorem~\ref{th:cumul} are satisfied.
Then the tail of the probability distribution $E_t$ of the process $\varphi(t)$ in (\ref{phi_psi}) admits an upper bound
\begin{equation}
\label{cramer2}
    \sup_{t>0}
    \Big(
    \frac{1}{t}
    \ln
    E_t([\eps t, +\infty))
    \Big)
    \<
    \frac{n\mu}{4}
    \Big(
        2 - \frac{n\alpha}{\eps} - \frac{\eps}{n\alpha}
    \Big),
    \qquad
    \eps \> n\alpha,
\end{equation}
where
\begin{equation}
\label{alpha}
    \alpha
    :=
    \|\sqrt{\Pi} \sqrt{\Gamma}\|
    \|\Gamma^{-1/2}(P+i\Theta)\sqrt{\Pi}\|.
\end{equation}
Here, $(\mu,\Gamma)$ is any pair of a scalar $\mu>0$ and a real positive definite symmetric matrix $\Gamma$ of order $n$  satisfying the algebraic Lyapunov inequality (ALI)
\begin{equation}
\label{ALI}
    A\Gamma + \Gamma A^{\rT}\preccurlyeq - 2\mu \Gamma.
\end{equation}

\end{theorem}
\begin{proof}
The pairs $(\mu,\Gamma)$, satisfying (\ref{ALI}), exist due to the matrix $A$ being Hurwitz. Any such pair satisfies the differential matrix inequality
\begin{equation}
\label{ALIdiff}
    \d_{\tau}
    \big(\re^{2\mu\tau}\re^{\tau A} \Gamma \re^{\tau A^{\rT}}\big)
    =
    \re^{2\mu\tau}
    \re^{\tau A} (A\Gamma + \Gamma A^{\rT} + 2\mu\Gamma) \re^{\tau A^{\rT}}\big)
    \preccurlyeq
    0.
\end{equation}
As a matrix-valued version of the Gronwall-Bellman lemma, the integration of (\ref{ALIdiff}) over $\tau$ leads to $    \re^{\tau A} \Gamma \re^{\tau A^{\rT}}
    \preccurlyeq
    \re^{-2\mu\tau}\Gamma
$, which is equivalent to the contraction property
\begin{equation}
\label{GBup}
    \|\Gamma^{-1/2} \re^{\tau A} \sqrt{\Gamma}\|
    \<
    \re^{-\mu\tau},
    \qquad
    \tau\>0.
\end{equation}
A combination of (\ref{GBup}) with (\ref{S}) and the submultiplicativity of the operator norm allows the exponential decay  of the function (\ref{N}) to be quantified as
\begin{align}
\nonumber
    N(\tau)
    & =
    \|\sqrt{\Pi} \sqrt{\Gamma} \Gamma^{-1/2} \re^{\tau A} \sqrt{\Gamma} \Gamma^{-1/2}(P+i\Theta)\sqrt{\Pi}\|\\
\nonumber
    & \<
    \|\sqrt{\Pi} \sqrt{\Gamma}\|
    \|\Gamma^{-1/2} \re^{\tau A} \sqrt{\Gamma}\|
    \|\Gamma^{-1/2}(P+i\Theta)\sqrt{\Pi}\|\\
\label{Nup}
    & \<
    \alpha
    \re^{-\mu\tau}
\end{align}
for all $\tau\>0$, where $\alpha$ is given by (\ref{alpha}).   Since the function $N$ is even, the inequality (\ref{Nup}) can be extended to all $\tau$ as
\begin{equation}
\label{Nup1}
    N(\tau) \< \alpha \re^{-\mu |\tau|}=:\wh{N}(\tau),
    \qquad
    \tau \in \mR.
\end{equation}
The right-hand side of (\ref{Nup1}) has the Fourier transform
\begin{equation}
\label{Fup}
    \wh{F}(\lambda)
    :=
    \int_{-\infty}^{+\infty}
    \wh{N}(\tau)
    \re^{-i\lambda \tau}
    \rd \tau
    =
    \frac{2\alpha \mu}{\lambda^2 + \mu^2},
\end{equation}
which is a rational function of $\lambda$ with the $L_{\infty}$-norm
\begin{equation}
\label{Fnorm}
  \|\wh{F}\|_{\infty}
  =
  \frac{2\alpha}{\mu}.
\end{equation}
The corresponding integral in (\ref{cramer1}) can be  evaluated by using a contour integration or through the parametric differentiation
\begin{align}
\nonumber
    -
    \d_{\theta}
    \int_{-\infty}^{+\infty}
    \ln
    (1 - 2\theta \wh{F}(\lambda))
    \rd
    \lambda
    & =
    2
    \int_{-\infty}^{+\infty}
    \frac{\wh{F}(\lambda)}{1 - 2\theta \wh{F}(\lambda)}
    \rd
    \lambda    \\
\nonumber
    &=
    4 \alpha \mu
    \int_{-\infty}^{+\infty}
    \frac{1}{\lambda^2 + \mu^2 - 4\theta \alpha \mu}
    \rd
    \lambda    \\
\label{diffpar}
    & =
    \frac{4\pi\alpha \mu}{\sqrt{\mu^2 - 4\theta \alpha \mu}}.
\end{align}
After integration over $\theta$, this yields
\begin{equation}
\label{intpar}
-
    \int_{-\infty}^{+\infty}
    \ln
    (1 - 2\theta \wh{F}(\lambda))
    \rd
    \lambda
    =
    2\pi
    \big(
        \mu - \sqrt{\mu^2 - 4\theta \alpha \mu}
    \big)
\end{equation}
for any $\theta$ satisfying
\begin{equation}
\label{thetarange1}
    0 \< \theta < \frac{1}{2\|\wh{F}\|_{\infty}} = \frac{\mu}{4\alpha}
\end{equation}
in view of (\ref{thetarange}) and (\ref{Fnorm}). By appropriately modifying (\ref{cramer1})--(\ref{epsbig})  for the functions $\wh{N}$ and $\wh{F}$ in (\ref{Nup1}) and (\ref{Fup}), and using (\ref{diffpar}) and (\ref{intpar}), it follows that
\begin{align*}
    \sup_{t>0}
    \Big(
    \frac{1}{t}
    \ln
    E_t([\eps t, +\infty))
    \Big)
    & \<
    \inf_{0\< \theta< \frac{1}{2\|\wh{F}\|_{\infty}}}
    \Big(
    -\frac{n}{4\pi}
    \int_{-\infty}^{+\infty}
    \ln
    (1 - 2\theta \wh{F}(\lambda))
    \rd
    \lambda
    -
        \theta \eps
    \Big)\\
    & =
    \inf_{0\< \theta< \frac{\mu}{4\alpha}}
    \Big(
        \frac{n}{2}
    \big(
        \mu - \sqrt{\mu^2 - 4\theta \alpha \mu}
    \big)
    -
    \theta \eps
    \Big)\\
    & =
        \frac{n\mu}{2}
    \Big(
        1 - \frac{n\alpha}{\eps}
    \Big)
    -
    \frac{\mu\eps}{4\alpha}\Big(1-\Big(\frac{n\alpha}{\eps}\Big)^2\Big)\\
    & =
    \frac{n\mu}{4}
    \Big(
        2 - \frac{n\alpha}{\eps} - \frac{\eps}{n\alpha}
    \Big),
\end{align*}
where the minimum over the interval (\ref{thetarange1}) is achieved at $\theta = \frac{\mu}{4\alpha}(1-(\frac{n\alpha}{\eps})^2)$, provided $\eps\> n\alpha$, thus establishing (\ref{cramer2}).
\end{proof}

Note that the ALI (\ref{ALI}) has a positive definite solution $\Gamma$ for any $\mu<-\max \{\Re \lambda:\ \lambda \in \fS\}$, where $\fS$ is the spectrum of the Hurwitz matrix $A$. A particular choice of such a matrix $\Gamma$ (which can be carried out, for example, by using the eigenbasis of $A$ in the case when $A$ is diagonalizable) affects the constant $\alpha$ in (\ref{alpha}) which enters (\ref{cramer2}) together with $\mu$.

\section{Classical covariance correspondence}\label{sec:CCCP}

We will now discuss a correspondence between vectors of self-adjoint quantum variables and auxiliary classical random vectors at the level of the second-order  moments.
More precisely, let $\zeta$ be a classical $\mC^n$-valued random vector given by
\begin{equation}
\label{zeta}
    \zeta:=\xi+i\eta,
    \qquad
    \xi := \Re \zeta,
    \qquad
    \eta:=\Im \zeta.
\end{equation}
Here, $\xi$ and $\eta$ are classical random vectors with values in $\mR^n$, which are also assembled into the $\mR^{2n}$-valued vector
\begin{equation}
\label{xieta}
    \vartheta
    :=
    \begin{bmatrix}
        \xi\\
        \eta
    \end{bmatrix}.
\end{equation}
Suppose $\zeta$ has zero mean and is square integrable, that is, $\bE\zeta = 0$ and $\bE(|\zeta|^2)<+\infty$, where $|\zeta|^2 = |\xi|^2 + |\eta|^2$. Then its covariance matrix is given by
\begin{align}
\nonumber
    \cov(\zeta)
    & =
    \bE (\zeta\zeta^*)
    =
    (\begin{bmatrix}
        1 & i
    \end{bmatrix}
    \ox I_n)
    \cov(\vartheta)
    \Big(\begin{bmatrix}
        1 \\ -i
    \end{bmatrix}
    \ox
    I_n\Big)\\
\nonumber
    & =
    \bE(\xi\xi^{\rT}) +
    \bE(\eta\eta^{\rT})
    -
    i
    (\bE(\xi\eta^{\rT}) -
    \bE(\eta\xi^{\rT}))\\
\label{covzeta}
    & =
    \cov(\xi)+
    \cov(\eta) -
    2i\bA(\cov(\xi,\eta)),
\end{align}
where $\bA(M):= \frac{1}{2}(M-M^{\rT})$  denotes the antisymmetrizer of square matrices which is applied here to the cross-covariance matrix  $\cov(\xi,\eta) = \bE(\xi\eta^{\rT})$ of the zero-mean vectors $\xi$ and $\eta$. Of particular interest for our purposes is the case when $\xi$ and $\eta$ in (\ref{xieta}) have identical covariance matrices and an antisymmetric cross-covariance matrix:
\begin{equation}
\label{Exieta}
    \cov(\vartheta)
    =
    \begin{bmatrix}
        \bE(\xi\xi^{\rT}) & \bE(\xi\eta^{\rT}) \\
        \bE(\eta\xi^{\rT})  & \bE(\eta\eta^{\rT})
    \end{bmatrix}
    =
    \frac{1}{2}
    \begin{bmatrix}
        P & -\Theta\\
        \Theta & P
    \end{bmatrix}.
\end{equation}
Here, $P$ is a real positive semi-definite symmetric matrix of order $n$, and $\Theta$ is a real antisymmetric matrix of order $n$. In this case, (\ref{covzeta}) takes the form
\begin{equation}
\label{covzeta1}
    \cov(\zeta) = P+i\Theta,
\end{equation}
where $P+i\Theta$ is isospectral (up to multiplicity of eigenvalues) to the matrix ${\small     \begin{bmatrix}
        P & -\Theta\\
        \Theta & P
    \end{bmatrix}}$ on the right-hand side of (\ref{Exieta}), with both matrices being positive semi-definite.
Now, let $\zeta$ depend on time and form  a $\mC^n$-valued  Markov diffusion process governed by a classical SDE
\begin{equation}
\label{dzeta}
    \rd \zeta = A \zeta \rd t + \frac{1}{\sqrt{2}}B\Omega \rd \omega
\end{equation}
driven by a standard Wiener process $\omega$ in $\mR^m$. Here, $\Omega$ is the matrix from (\ref{WW}),  and the matrices $A$ and $B$ are  given by (\ref{A_B}) as before. Since both $A$ and $B$ are real, then, due to the structure of $\Omega$, the SDE (\ref{dzeta}) can be represented in terms of the real and imaginary parts in (\ref{zeta}) as
\begin{equation}
\label{dxi_deta}
    \rd \xi
     = A \xi \rd t + \frac{1}{\sqrt{2}}B\rd \omega,
     \qquad
    \rd \eta
     = A \eta \rd t + \frac{1}{\sqrt{2}}BJ \rd \omega.
\end{equation}
Equivalently, the augmented process $\vartheta$ in (\ref{xieta}) satisfies the SDE
\begin{align}
\label{dxieta}
    \rd \vartheta
    & = (I_2\ox A)\vartheta \rd t + \frac{1}{\sqrt{2}}(I_2 \ox B)\begin{bmatrix}I_m\\ J\end{bmatrix}\rd \omega,
\end{align}
The following lemma establishes a correspondence between the classical Markov diffusion processes $\xi$, $\eta$, $\zeta$,  and the linear quantum stochastic system considered in the previous sections.

\begin{lemma}
\label{lem:CCCP}
Suppose the matrices $A$ and $B$ are given by (\ref{A_B}), with $A$ being Hurwitz.
Then the classical Markov diffusion process $\vartheta$ in (\ref{xieta}), governed by (\ref{dxieta}), has a unique invariant measure (the invariant joint distribution of $\xi$ and $\eta$ in (\ref{dxi_deta})) which is a Gaussian distribution in $\mR^{2n}$ with zero mean and covariance matrix (\ref{Exieta}), where $P$ is given by (\ref{P}) and $\Theta$ is the CCR matrix from (\ref{Theta}).  The corresponding invariant measure of the process $\zeta$ in (\ref{zeta}) and (\ref{dzeta}) is a Gaussian distribution in $\mC^n$ with zero mean and the covariance matrix (\ref{covzeta1}).
Furthermore, if $\vartheta$ and the quantum process $X$ are initialised at the  corresponding invariant Gaussian states, then their two-point covariance functions reproduce each other in the sense that
\begin{equation}
\label{EXXts}
    \bE (\zeta(t)\zeta(s)^*)
    =
    \bE (X(t)X(s)^{\rT})
    =
    \re^{(t-s)A}
    (P+i\Theta)
\end{equation}
for all $t\>s\>0$, with the classical expectation on the left-hand side and the quantum expectation on the right-hand side. \hfill$\square$
\end{lemma}
\begin{proof}
 Due to $A$ being Hurwitz,  the existence,  uniqueness and Gaussian nature of the invariant measure for the process $\vartheta$ follows from (\ref{dxieta}),  whereby the invariant distribution has zero mean and covariance matrix $\cov(\vartheta)$ satisfying the ALE
 \begin{equation}
 \label{SALE}
    (I_2\ox A)\cov(\vartheta)  + \cov(\vartheta)(I_2\ox A^{\rT})
      +  \frac{1}{2}(I_2 \ox B)\begin{bmatrix}I_m& -J\\ J & I_m\end{bmatrix}(I_2 \ox B^{\rT})
    =0,
 \end{equation}
 where use is made of the orthogonality of the real antisymmetric matrix $J$ in (\ref{J}). The ALE (\ref{SALE}) splits into three equations
 \begin{align}
 \label{xiALE}
 A\cov(\xi)+\cov(\xi)A^{\rT} + \frac{1}{2}BB^{\rT} & =0,\\
 \label{etaALE}
 A\cov(\eta)+\cov(\eta)A^{\rT} + \frac{1}{2}BB^{\rT} & =0,\\
 \label{xietaALE}
 A\cov(\xi,\eta)+\cov(\xi,\eta)A^{\rT} - \frac{1}{2}BJB^{\rT} & =0.
 \end{align}
Since the ALEs (\ref{xiALE}) and (\ref{etaALE}) are identical to (\ref{PALE}) up to a factor of $\frac{1}{2}$ (and have unique solutions due to the matrix $A$ being Hurwitz),  it follows that
 \begin{equation}
 \label{covcovP}
    \cov(\xi) = \cov(\eta) = \frac{1}{2}P,
 \end{equation}
with $P$ given by (\ref{P}). By a similar reasoning,  a comparison of (\ref{xietaALE}) with the PR condition (\ref{APR}) leads to
 \begin{equation}
 \label{covcovTheta}
    \cov(\xi,\eta) = -\frac{1}{2}\Theta,
 \end{equation}
 where $\Theta$ is the CCR matrix from (\ref{Theta}).
The relations (\ref{covcovP}) and (\ref{covcovTheta}) imply that the covariance matrix of the invariant joint Gaussian  distribution of the processes $\xi$ and $\eta$ is indeed given by (\ref{Exieta}), with (\ref{covzeta1}) describing the corresponding covariance matrix for $\zeta$. Now, the linear SDE (\ref{dzeta}) implies that
\begin{equation}
\label{zetats}
    \zeta(t)
    =
    \re^{(t-s)A}\zeta(s) + \frac{1}{\sqrt{2}}\int_s^t \re^{(t-\tau)A}B\Omega \rd \omega(\tau)
\end{equation}
for all $t\> s\> 0$.  Due to $\omega$ being a standard Wiener process independent of the initial value $\zeta(0)$, it follows from (\ref{zetats}) that
\begin{equation}
\label{EXXts2}
    \bE (\zeta(t)\zeta(s)^*)
    =
    \re^{(t-s)A}\bE (\zeta(s)\zeta(s)^*).
\end{equation}
If the quantum system is initialised at the invariant Gaussian state and $\vartheta(0)$ has the corresponding invariant Gaussian distribution, then
\begin{equation}
\label{EXXts3}
    \cov(\zeta(s)) = P+i\Theta = \cov(X(s))
\end{equation}
for all $s\>0$. Substitution of (\ref{EXXts3}) into (\ref{Xcovts}) and (\ref{EXXts2}) leads to (\ref{EXXts}).
\end{proof}

Note that Lemma~\ref{lem:CCCP}  is concerned only with the correspondence between classical Gaussian Markov diffusion  processes and quantum system variables in Gaussian quantum states at the level of covariances. This correspondence involves complex conjugation and does not extend, in general, from the second-order moments to arbitrary higher-order moments. For example,
\begin{equation}
\label{zPiz}
    \zeta^* \Pi\zeta
    =
    \xi^{\rT} \Pi\xi + \eta^{\rT} \Pi \eta
\end{equation}
is a real-valued classical random process which has the same steady-state mean value as the quantum process $\psi$ in (\ref{Epsi}):
$$
    \bE(\zeta^* \Pi\zeta)
    =
    \bra \Pi, \bE(\zeta \zeta^*)\ket
    =
    \bra \Pi, P+i\Theta\ket
    =
    \bra \Pi, P\ket,
$$
which follows from (\ref{EXXts}). At the same time, application of the classical covariance relation (\ref{covquad}) to (\ref{zPiz}) leads to
\begin{align}
\nonumber
    \var(\zeta^* \Pi\zeta )
    & =
    \var(\xi^{\rT} \Pi\xi)
    +
    \var(\eta^{\rT} \Pi \eta)
    +
    2\cov(\xi^{\rT} \Pi\xi, \eta^{\rT} \Pi\eta)\\
\nonumber
    & =
        2
    \Bra
        \Pi,
        \cov(\xi)
        \Pi
        \cov(\xi)
        +
        \cov(\eta)
        \Pi
        \cov(\eta)
        +
        2
        \cov(\xi,\eta)
        \Pi
        \cov(\eta,\xi)
    \Ket\\
\label{varzPiz}
    & =
    \Bra
        \Pi,
        P\Pi P - \Theta \Pi \Theta
    \Ket,
\end{align}
where use is also made of (\ref{covcovP}) and (\ref{covcovTheta}). Note that the right-hand side of (\ref{varzPiz}) is different  from its quantum counterpart
$$
    \var(X^{\rT}\Pi X)
    =
    2
        \Bra
            \Pi,
            P\Pi P + \Theta \Pi \Theta
        \Ket
$$
which is obtained  by letting $\sigma=\tau$ in (\ref{covpsi}) and (\ref{S}). This discrepancy (which manifests itself already at the level of the fourth-order moments) is closely related to  the absence of classical joint probability distributions for noncommuting quantum variables. Nevertheless, of interest is further comparison of the QEF (\ref{QEF}) with the corresponding risk-sensitive cost functional for the SDE (\ref{dzeta}) whose asymptotic behaviour is described by
\begin{align}
\nonumber
    \lim_{t\to +\infty}
    \Big(
    \frac{1}{t}
    \ln
    \bE \re^{\theta \int_0^t \zeta(s)^*\Pi \zeta(s)\rd s}
    \Big)
    & =
    -
    \frac{1}{4\pi}
    \int_{-\infty}^{+\infty}
    \ln\det(I_n-\theta \Pi G(i\lambda) \Omega G(i\lambda)^*)
    \rd \lambda\\
\label{class}
    & =
    \frac{1}{4\pi}
    \sum_{r=1}^{+\infty}
    \frac{\theta^r}{r}
    \int_{-\infty}^{+\infty}
    \Tr
    \big(
        (\Pi D(\lambda))^r
    \big)
    \rd \lambda.
\end{align}
Here, $G$ is the transfer function in (\ref{G}),
so that $\frac{1}{2}G(i\lambda) \Omega G(i\lambda)^*$ is the spectral density of the process $\zeta$ in (\ref{dzeta}) which is expressed in terms of the spectral density $D$ for the quantum covariance function $S$ in (\ref{D}) and (\ref{DG}). In the case $\Pi\succcurlyeq 0$, the right-hand side of  (\ref{class}) is well-defined if the risk sensitivity parameter $\theta>0$ is bounded in terms of a weighted $\cH_{\infty}$-norm of the transfer function $G$ as
$
    \theta < \frac{2}{\|\sqrt{\Pi} G \Omega\|_{\infty}^2}
$,
where we have also used the property $\Omega^2 = 2\Omega$ of the matrix $\Omega$ from (\ref{WW}). Note that $    \int_{-\infty}^{+\infty}
    \Tr
    \big(
        (\Pi D(\lambda))^r
    \big)
    \rd \lambda$ in (\ref{class}) is a classical counterpart of the weighted sum of integrals on the right-hand side of (\ref{bKphiasy}).

\section{A numerical example of the quartic approximation of the QEF and large deviations bounds}\label{sec:numer}

We will now demonstrate the computation of the quartic  approximation  of the QEF from Section~\ref{sec:quart} (in the infinite-horizon limit)  for a two-mode OQHO with $n=4$ driven by $m=4$ external fields. It is assumed that the CCR matrix in (\ref{Theta}) is given by
\begin{equation}
\label{Thetaex}
    \Theta
    =
    \frac{1}{2}\bJ\ox I_2,
\end{equation}
which, in accordance with (\ref{ThetaJ}) and (\ref{bJ}), corresponds to the system variables $q_1, q_2, p_1, p_2$ consisting of the conjugate quantum mechanical positions and momenta.
The energy and coupling matrices of the OQHO were randomly generated so as to make the matrix $A$  in (\ref{A_B}) Hurwitz:
\begin{equation}
\label{RMex}
    R
     =
    {\small\begin{bmatrix}
   -0.1027 &   1.3449 &  -0.2403 &  -1.3994\\
    1.3449 &   1.5008 &  -0.1856 &   0.9212\\
   -0.2403 &  -0.1856 &  -0.5704 &  -0.4146\\
   -1.3994 &   0.9212 &  -0.4146 &  -0.3233
    \end{bmatrix}},
    \qquad
    M
    =
    {\small\begin{bmatrix}
    0.8726 &   0.1632 &   2.1844 &  -1.9270\\
    0.1179 &  -0.8147 &  -0.0938 &   0.5214\\
   -1.5031 &   0.4037 &  -0.2942 &  -2.0544\\
    0.9218 &   0.7562 &  -0.5048 &  -0.2698
    \end{bmatrix}}
\end{equation}
(the resulting $A$ has the eigenvalues $-0.5532 \pm 2.5929i$,
  $-1.3302$, $-4.2068$). The positive definite weighting matrix $\Pi$, specifying the QEF through (\ref{phi_psi})--(\ref{QEF}), was also randomly generated:
\begin{equation}
\label{Piex}
\Pi
    =
    {\small\begin{bmatrix}
    3.5050 &  -0.5447 &   0.0672 &  -2.3918\\
   -0.5447 &   4.0758 &  -1.1876 &   0.0215\\
    0.0672 &  -1.1876 &   5.1422 &  -1.4628\\
   -2.3918 &   0.0215 &  -1.4628 &   4.5416
   \end{bmatrix}}.
\end{equation}
The infinite-horizon controllability Gramian $P$ in (\ref{P}) and the matrix $T$ in (\ref{T}) were found by successively solving the ALEs (\ref{PALE}) and (\ref{TALE}):
$$
P
    =
    {\small\begin{bmatrix}
    3.7981 &  -2.5143 &  -3.8716 &  -1.6214\\
   -2.5143 &   4.9443 &   0.5356 &   0.4305\\
   -3.8716 &   0.5356 &   6.7086 &   2.8509\\
   -1.6214 &   0.4305 &   2.8509 &   1.4473
   \end{bmatrix}},
   \qquad
    T
    =
    {\small\begin{bmatrix}
  131.5431 &-108.9564 &-138.4442 & -58.4033\\
 -108.9564 & 138.7545 &  60.4808 &  21.2105\\
 -138.4442 &  60.4808 & 204.6153 &  91.4998\\
  -58.4033 &  21.2105 &  91.4998 &  41.2158
   \end{bmatrix}}.
$$
The corresponding asymptotic growth rates for the first two cumulants of the quantum process $\varphi$ in (\ref{Ephi}) and (\ref{Varphiasy}) took the values
$$
    \lim_{t\to +\infty}
    \Big(
    \frac{1}{t}
    \bE \varphi(t)
    \Big)
     =
    74.9147,
    \qquad
   \lim_{t\to +\infty}
    \Big(
    \frac{1}{t}
    \var(\varphi(t))
    \Big)
    =
    8.9399\x 10^3.
$$
In this example, the threshold value for the risk-sensitivity parameter $\theta$ in (\ref{theta0}), within which the quartic term in (\ref{F}) can be neglected, is
$$
    \theta_0
    =
    0.0168.
$$
Since $\theta_0$ is small, the coefficients $\frac{\theta^k}{k!}$  of the higher-order moments $\bE (\varphi(t)^k)$ and cumulants $\bK_k(\varphi(t))$  in (\ref{Ximom}) and (\ref{expquad1}) (with $k\> 3$) are negligible compared to $\theta^2$ for such values of $\theta$. This makes the quartic approximation (\ref{F}) a reasonable approximation of the QEF  for the range of small values of $\theta$ in this example. Since the matrix $A$ under consideration is diagonalizable, the ALI (\ref{ALI}) is satisfied with $\mu = 
0.5532$ (the negative of the largest real part of its eigenvalues) 
and
$$
    \Gamma
    =
    UU^*
    =
    {\small\begin{bmatrix}
    1.4750 &   -0.4852 &  -1.4090 &  -0.2636\\
   -0.4852 &    0.6271 &   0.2354 &   0.1475\\
   -1.4090 &    0.2354 &   1.6303 &   0.3569\\
   -0.2636 &   0.1475  &  0.3569  &  0.2676
    \end{bmatrix}   },
$$
where the matrix $U$ is formed from the eigenvectors of $A$.  The corresponding constant (\ref{alpha}) is
$\alpha = 69.6784$,  and the large deviations bound (\ref{cramer2}) of Theorem~\ref{th:devup} is shown in Fig.~\ref{fig:num}.
\begin{figure}[thpb]
      \centering
      \includegraphics[width=120mm]{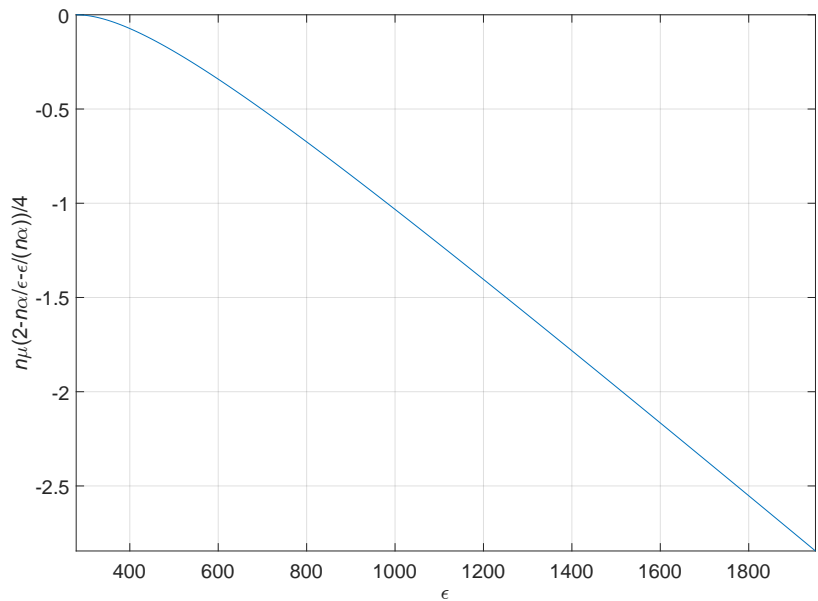}
      \caption{The large deviations probability bound (\ref{cramer2}) as a function of the scale parameter $\eps\> n\alpha$ for the two-mode OQHO of the example specified by (\ref{Thetaex})--(\ref{Piex}).}
      \label{fig:num}
   \end{figure}

\section{Conclusion}\label{sec:conc}

For an  open quantum harmonic oscillator, governed by a linear QSDE driven by vacuum bosonic fields, we have considered a quadratic-exponential  functional which penalizes the second and higher-order moments of the system variables. We have obtained an integro-differential equation for the time evolution of the QEF and compared it with the original quantum risk-sensitive performance criterion which was used previously in measurement-based quantum control and filtering problems. We have discussed multi-point Gaussian quantum states for the system variables at different instants and employed their first four moments for a quartic approximation of the QEF whose infinite-horizon asymptotic behaviour has also been investigated. A numerical example has been provided in order to demonstrate the approximation. Higher-order cumulants, associated with the QEF, have been related to combinatorics of consecutive inversions in permutations and gives rise to a nontrivial problem of their recursive computation. They have also been used for a Cramer-type large deviations estimate for the system variables. We have also considered an auxiliary classical Gaussian Markov diffusion  process in a complex Euclidean space, which reproduces the quantum system variables at the level of covariance functions but has different higher-order moments relevant to the risk-sensitive criteria. The results of the paper may find applications to coherent (measurement-free) quantum risk-sensitive control problems, where the plant and controller form a fully quantum closed-loop system, and other settings with nonquadratic cost functionals.

\appendix
\section{Commutator of quadratic polynomials of quantum variables satisfying CCRs}
\label{sec:quadcom}
\renewcommand{\theequation}{A\arabic{equation}}
\setcounter{equation}{0}

Since the commutator $\ad_{\xi}$, associated with a given operator $\xi$, is a derivation, then repeated application of this property leads to
\begin{align}
\nonumber
    [ab,c d]
    & =
    [ab,c]d+ c[ab,d]\\
\nonumber
    & =
    ([a,c]b + a[b,c])d
    +
    c([a,d]b + a[b,d])\\
\label{abcdapp}
    & =
    [a,c]b d + a[b ,c]d
    +c[a,d]b  + ca[b ,d]
\end{align}
for any operators $a$, $b$, $c$, $d$. Hence, if  these operators satisfy CCRs (that is, their pairwise commutators are identity operators up to scalar factors), then the right-hand side of (\ref{abcdapp}) is a quadratic polynomial of $a$, $b$, $c$, $d$. Therefore, quadratic polynomials of such operators form a Lie algebra with respect to the commutator. For the purposes of Section~\ref{sec:QEF}, we will provide a version of (\ref{abcdapp}) for quadratic polynomials of operators satisfying the CCRs.

\begin{lemma}
\label{lem:quadcom}
Suppose $a:=(a_j)$, $b:=(b_k)$, $c:=(c_{\ell})$, $d:=(d_r)$ are vectors of self-adjoint operators which satisfy  the CCRs
\begin{equation}
\label{abcdCCR}
    \left[
    \begin{bmatrix}
        a\\
        b
    \end{bmatrix},
    \begin{bmatrix}
        c^{\rT}&
        d^{\rT}
    \end{bmatrix}
    \right]
    :=
    \begin{bmatrix}
        [a,c^{\rT}] & [a,d^{\rT}]\\
        [b,c^{\rT}] & [b,d^{\rT}]
    \end{bmatrix}
    =
    2i
   \begin{bmatrix}
     \Theta_{11} & \Theta_{12} \\
     \Theta_{21} & \Theta_{22}
   \end{bmatrix},
\end{equation}
where $\Theta_{jk}$ are real matrices. Also, let $F:= (f_{jk})$ and $G:=(g_{\ell r})$ be appropriately dimensioned complex matrices which specify the bilinear forms
\begin{equation}
\label{aFb_cGd}
    a^{\rT}F b
     =\sum_{j,k}f_{jk}a_jb_k,
     \qquad
    c^{\rT} G d
    =
    \sum_{\ell,r}g_{\ell r}c_{\ell}d_r.
\end{equation}
Then their commutator is also a quadratic polynomial of the operators which is computed as
\begin{equation}
\label{quadcom}
    [a^{\rT}F b, c^{\rT}G d]
    =
    2i
    \begin{bmatrix}
        a^{\rT}&
        b^{\rT}&
        c^{\rT}
    \end{bmatrix}
    \begin{bmatrix}
      0 & 0 & F \Theta_{21} G \\
      0 & 0 & F^{\rT} \Theta_{11} G \\
      G \Theta_{22}^{\rT} F^{\rT} & G \Theta_{12}^{\rT} F & 0
    \end{bmatrix}
    \begin{bmatrix}
        a \\
        b \\
        d
    \end{bmatrix}.
\end{equation}
In the case
$a=b$ (when the vectors $a$ and $b$ are identical), (\ref{quadcom}) takes the form
\begin{equation}
\label{quadcom2}
    [a^{\rT}F a,c^{\rT}G d]
    =
    4i
    \begin{bmatrix}
        a^{\rT}&
        c^{\rT}
    \end{bmatrix}
    \begin{bmatrix}
      0 & \bS(F)\Theta_{11} G\\
      G \Theta_{12}^{\rT} \bS(F) & 0
    \end{bmatrix}
    \begin{bmatrix}
        a \\
        d
    \end{bmatrix},
\end{equation}
where $\bS(F):= \frac{1}{2}(F+F^{\rT})$  denotes the symmetrizer of square matrices.
\hfill$\square$
\end{lemma}
\begin{proof}
By recalling the bilinearity of the commutator, applying (\ref{abcdapp}) to the operators $a_j$, $b_k$, $c_{\ell}$, $d_r$ in (\ref{aFb_cGd}) and using the CCRs (\ref{abcdCCR}), it follows that
\begin{align}
\nonumber
    [a^{\rT}F b,c^{\rT}G d]
    & =
    \sum_{j,k,\ell,r}
    f_{jk}g_{\ell r}
    [a_jb_k, c_{\ell}d_r]\\
\nonumber
    & =
    \sum_{j,k,\ell,r}
    f_{jk}g_{\ell r}
    \big(
        [a_j,c_{\ell}]b_k d_r + a_j[b_k ,c_{\ell}]d_r+c_{\ell}[a_j,d_r]b_k  + c_{\ell}a_j[b_k ,d_r]
    \big)        \\
\label{quadcom1}
    & =
    2i
    \big(
        b^{\rT} F^{\rT} \Theta_{11} Gd
        +
        a^{\rT} F \Theta_{21} Gd
    +
    c^{\rT}G \Theta_{12}^{\rT} Fb
    +
    c^{\rT}G \Theta_{22}^{\rT} F^{\rT}a
    \big).
\end{align}
The right-hand side of (\ref{quadcom1}) is a quadratic function of the operators whose vector-matrix form is given by (\ref{quadcom}). If $a=b$, then $F$ in (\ref{aFb_cGd}) is a square matrix, and the matrices in (\ref{abcdCCR}) satisfy $\Theta_{1k}=\Theta_{2k}$ for every $k=1,2$. In this case,
(\ref{quadcom1}) reduces to
\begin{align}
\nonumber
    [a^{\rT}F a,c^{\rT}G d]
    &=
    2i
    \big(
        a^{\rT} (F^{\rT} \Theta_{11}+F \Theta_{21}) Gd+
    c^{\rT}G (\Theta_{12}^{\rT} F+\Theta_{22}^{\rT} F^{\rT})a
    \big)\\
\nonumber
    &=
    2i
    \big(
        a^{\rT} (F^{\rT} +F )\Theta_{11} Gd
    +
    c^{\rT}G \Theta_{12}^{\rT}( F+F^{\rT})a
    \big)    \\
\nonumber
    &=
    4i
    \big(
        a^{\rT} \bS(F)\Theta_{11} Gd
    +
    c^{\rT}G \Theta_{12}^{\rT}\bS(F)a
    \big),
\end{align}
which establishes (\ref{quadcom2}).
\end{proof}

Similar commutation relations hold for bilinear forms of annihilation and creation operators (see, for example, \cite[Appendix B]{JK_1998} and \cite[Lemma 4.2]{PUJ_2012}) and are used in the context of Schwinger's theorems on exponentials of such forms \cite{CG_2010}.

\section{Covariance of quadratic functions of Gaussian quantum variables}
\label{sec:covquad}
\renewcommand{\theequation}{B\arabic{equation}}
\setcounter{equation}{0}

For the purposes of Section~\ref{sec:quart}, we will need the following lemma on the covariance of bilinear forms of Gaussian quantum variables.

\begin{lemma}
\label{lem:covquad}
Suppose the self-adjoint quantum variables, constituting the vectors $a$, $b$, $c$, $d$ in Lemma~\ref{lem:quadcom},  are in a zero-mean Gaussian state. Then the covariance of the bilinear forms in (\ref{aFb_cGd}), specified by complex matrices $F$ and $G$, can be computed as
\begin{equation}
\label{covquad0}
    \cov(
        a^{\rT}F b,
        c^{\rT}G d
    )
    =
    \Bra
        \overline{F\cov(b,d)},
        \cov(a,c)G
    \Ket
    +
    \Bra
        \overline{F\cov(b,c)},
        \cov(a,d)G^{\rT}
    \Ket.
\end{equation}
In the case when $a=b$ and $c=d$, the relation (\ref{covquad0}) reduces to
\begin{equation}
\label{covquad1}
    \cov(
        a^{\rT}F a,
        c^{\rT}G c
    )
    =
    2
    \Bra
        \overline{\bS(F)},
        \cov(a,c)
        \bS(G)
        \cov(a,c)^{\rT}
    \Ket,
\end{equation}
where $\bS$ is the symmetrizer.
\hfill$\square$
\end{lemma}
\begin{proof}
The mean value of the first bilinear form in (\ref{aFb_cGd}) is computed as
\begin{equation}
\label{aFbmean}
    \bE(a^{\rT}F b)
    =
    \sum_{j,k}
    f_{jk}
    \bE (a_jb_k)
    =
    \Tr(
        F^{\rT}
        \bE(ab^{\rT})
    )
     =
    \Bra
        \overline{F},
        \cov(a,b)
    \Ket,
\end{equation}
where $\bE(ab^{\rT}) = \cov(a,b)$ since the underlying quantum variables are assumed to have zero mean values.
By a similar reasoning, the second bilinear form in (\ref{aFb_cGd}) has the mean value
\begin{align}
\label{cGdmean}
    \bE(c^{\rT}G d)
    & =
    \Bra
        \overline{G},
        \cov(c,d)
    \Ket.
\end{align}
Application of the Wick-Isserlis theorem \cite[Theorem 1.28 on pp.~11--12]{J_1997} (see also \cite{I_1918} and \cite[p.~122]{M_2005}) to the fourth-order mixed moment of the zero-mean Gaussian quantum variables $a_j$, $b_k$, $c_{\ell}$, $d_r$ leads to
\begin{align}
\nonumber
    \bE(
        a^{\rT}F b
        c^{\rT}G d
    )
    & =
    \sum_{j,k,\ell,r}
    f_{jk}g_{\ell r}
    \bE(a_j b_k c_{\ell} d_r)\\
\nonumber
    & =
    \sum_{j,k,\ell,r}
    f_{jk}g_{\ell r}
    \big(
        \cov(a_j,b_k) \cov(c_{\ell},d_r)
        +
        \cov(a_j,c_{\ell}) \cov(b_k,d_r)
        +
        \cov(a_j,d_r) \cov(b_k,c_{\ell})
    \big)\\
\nonumber
    & =
    \Tr(F^{\rT}\cov(a,b))
    \Tr(G^{\rT}\cov(c,d))
    +
    \Tr((F\cov(b,d))^{\rT} \cov(a,c)G)
    +
    \Tr((F\cov(b,c))^{\rT} \cov(a,d)G^{\rT})\\
\label{EaFbcGd}
    & =
    \Bra
        \overline{F},
        \cov(a,b)
    \Ket
    \Bra
        \overline{G},
        \cov(c,d)
    \Ket
    +
    \Bra
        \overline{F\cov(b,d)},
        \cov(a,c)G
    \Ket
    +
    \Bra
        \overline{F\cov(b,c)},
        \cov(a,d)G^{\rT}
    \Ket.
\end{align}
The
covariance of the bilinear forms (\ref{aFb_cGd}) can now be computed
by combining (\ref{aFbmean})--(\ref{EaFbcGd}) as
\begin{align*}
    \cov(
        a^{\rT}F b,
        c^{\rT}G d
    )
    & =
    \bE(
        a^{\rT}F b
        c^{\rT}G d
    )
     -
    \bE(
        a^{\rT}F b
    )
    \bE(
        c^{\rT}G d
    )\\
    & =
    \Bra
        \overline{F\cov(b,d)},
        \cov(a,c)G
    \Ket
    +
    \Bra
        \overline{F\cov(b,c)},
        \cov(a,d)G^{\rT}
    \Ket,
\end{align*}
which establishes (\ref{covquad0}). Application of (\ref{covquad0}) to the particular case $a=c$ and $c=d$ leads to
\begin{align}
\nonumber
    \cov(
        a^{\rT}F a,
        c^{\rT}G c
    )
    & =
    \Bra
        \overline{F\cov(a,c)},
        \cov(a,c)G
    \Ket
    +
    \Bra
        \overline{F\cov(a,c)},
        \cov(a,c)G^{\rT}
    \Ket\\
\nonumber
    & =
    2
    \Bra
        \overline{F\cov(a,c)},
        \cov(a,c)\bS(G)
    \Ket\\
\nonumber
    & =
    2
    \Bra
        \overline{F},
        \cov(a,c)
        \bS(G)
        \cov(a,c)^{\rT}
    \Ket\\
\label{covquad2}
    & =
    2
    \Bra
        \overline{\bS(F)},
        \cov(a,c)
        \bS(G)
        \cov(a,c)^{\rT}
    \Ket,
\end{align}
thus proving (\ref{covquad1}). In (\ref{covquad2}) use has also been made of the symmetry of the matrix $        \cov(a,c)
        \bS(G)
        \cov(a,c)^{\rT}
$, the orthogonality of the subspaces of symmetric and antisymmetric matrices, and the fact that the symmetrizer commutes with the complex conjugation: $\bS(\overline{F} ) = \overline{\bS(F)}$.
\end{proof}

\section{An averaging lemma with multifactor convolutions}
\label{sec:conv}
\renewcommand{\theequation}{C\arabic{equation}}
\setcounter{equation}{0}

For the purposes of the cumulant growth rate results of Theorem~\ref{th:asycumul} in Section~\ref{sec:cumul}, we provide an averaging lemma which involves multifactor convolutions (\ref{conv}).

\begin{lemma}
\label{lem:conv}
Suppose $f_1, \ldots, f_r$ are appropriately dimensioned complex matrix-valued  functions on the real line which are bounded and absolutely integrable, with $r\> 2$. Then
\begin{align}
\nonumber
    \lim_{t\to +\infty}
    \Big(
    \frac{1}{t}
    \int_{[0,t]^r}
    \rprod_{k=1}^r
    f_k(t_k-t_{k+1})
    \rd t_1 \x \ldots \x \rd t_r
    \Big)
    & =
    (f_1* \ldots *f_r) (0)\\
\label{fint}
    & =
    \frac{1}{2\pi}
    \int_{-\infty}^{+\infty}
    \rprod_{k=1}^r
    F_k(\lambda)
    \rd \lambda,
\end{align}
where $t_{r+1} := t_1$, and   $F_k(\lambda):= \int_{-\infty}^{+\infty} f_k(t)\re^{-i\lambda t} \rd t$ denotes the Fourier transform of $f_k$.
\hfill$\square$
\end{lemma}
\begin{proof}
Since the integrand on the left-hand side of (\ref{fint}) depends on the integration variables $t_1, \ldots, t_r$ only through their differences, they can be translated so as to represent the integral in the form
\begin{equation}
\label{fg}
    \int_{[0,t]^r}
    \rprod_{k=1}^r
    f_k(t_k-t_{k+1})
    \rd t_1 \x \ldots \x \rd t_r
    =
    \int_0^t
    g_t(t_1)
    \rd t_1,
\end{equation}
where
\begin{align}
\nonumber
    g_t(t_1)
    & :=
    \int_{[0,t]^{r-1}}
    \rprod_{k=1}^r
    f_k(t_k-t_{k+1})
    \rd t_2 \x \ldots \x \rd t_r\\
\label{gs}
    & =
    \int_{[-t_1, t-t_1]^{r-1}}
    \rprod_{k=1}^r
    f_k(\tau_k-\tau_{k+1})
    \rd \tau_2 \x \ldots \x \rd \tau_r.
\end{align}
Here, use is made of the new integration variables
$
    \tau_k := t_k-t_1
$ for all $
    k=2, \ldots, r
$ together with the corresponding convention $\tau_1 = \tau_{r+1}=0$.
The right-hand side of (\ref{gs}) is organized as the convolution $(f_1* \ldots *f_r)(0)$ (evaluated at $0$) except that the integration is restricted to the cube $[-t_1, t-t_1]^{r-1}$.
Now, the fulfillment of the inclusions $t_1 \in [0, t]$ and $(\tau_2, \ldots, \tau_r) \in [-t_1, t-t_1]^{r-1}$ is equivalent to $t_1$ belonging to the (possibly empty) interval
\begin{equation}
\label{tint}
    -
    \min(0, \tau_2, \ldots, \tau_r)
    \<
    t_1
    \<
    t
    -
    \max(0, \tau_2, \ldots, \tau_r)
\end{equation}
of length
\begin{equation}
\label{phichi}
    \max
    \big(
        0,\,
        t
        -
        \max(0, \tau_2, \ldots, \tau_r)
        +
        \min(0, \tau_2, \ldots, \tau_r)
    \big)
    =
    t h_t(\tau_2, \ldots, \tau_r).
\end{equation}
Here,
\begin{equation}
\label{phi}
    h_t(\tau_2, \ldots, \tau_r)
    :=
    \chi_t
    \big(
        \max(0, \tau_2, \ldots, \tau_r)
        -
        \min(0, \tau_2, \ldots, \tau_r)
    \big)
\end{equation}
inherits from the function $\chi_t$ in (\ref{chi}) the properties of being bounded by and convergent to $1$ as $t\to +\infty$ for any given $\tau_2, \ldots, \tau_r \in \mR$. By combining (\ref{fg})--(\ref{phi}), it follows that
\begin{align}
\nonumber
    \frac{1}{t}
    \int_{[0,t]^r}
    \rprod_{k=1}^r
    f_k(t_k-t_{k+1})
    \rd t_1 \x \ldots \x \rd t_r
    & =
    \int_{\mR^{r-1}}
    h_t(\tau_2, \ldots, \tau_r)
    \rprod_{k=1}^r
    f_k(\tau_k-\tau_{k+1})
    \rd \tau_2 \x \ldots \x \rd \tau_r    \\
\label{fg1}
    &
    \to
    \int_{\mR^{r-1}}
    \rprod_{k=1}^r
    f_k(\tau_k-\tau_{k+1})
    \rd \tau_2 \x \ldots \x \rd \tau_r,
    \qquad
    {\rm as}\
    t\to +\infty,
\end{align}
where Lebesgue's dominated convergence theorem is applicable since the functions $f_1, \ldots, f_r$ are bounded and absolutely integrable (and hence, so are their convolutions). The limit in (\ref{fg1}) is $(f_1* \ldots *f_r)(0)$, which establishes the first of the equalities (\ref{fint}), with the second of them following from the convolution theorem.
\end{proof}

\end{document}